\documentclass[12pt,a4paper]{amsart}
\usepackage{amsmath}
\usepackage{amsthm}
\usepackage{amssymb}
\usepackage{drpack}
\usepackage[english]{babel}
\usepackage{verbatim}
\usepackage[shortlabels]{enumitem}
\usepackage{tikz-cd}

\newtheoremstyle{mio}%
{}{} 
{\itshape}{} 
{\bfseries}{.}{ } 
{#1 #2\thmnote{~\mdseries(#3)}} 
\theoremstyle{mio}
\newtheorem{teor}{Theorem}[section]
\newtheorem{cor}[teor]{Corollary}
\newtheorem{prop}[teor]{Proposition}
\newtheorem{lemma}[teor]{Lemma}
\newtheorem{defin}[teor]{Definition}

\newtheoremstyle{definition2}%
{}{} 
{}{} 
{\bfseries}{.}{ } 
{#1 #2\thmnote{\mdseries~ #3}} 
\theoremstyle{definition2}
\newtheorem{ex}[teor]{Example}
\newtheorem{oss}[teor]{Remark}

\theoremstyle{definition}
\newtheorem*{conj}{Conjecture}

\title{Free groups of ideals}
\author{Dario Spirito}
\address{Dipartimento di Scienze Matematiche, Fisiche e Informatiche, Universit\`a di Udine, Udine, Italy}
\email{dario.spirito@uniud.it}
\keywords{Invertible ideals; star operations; free groups; Pr\"ufer domains; locally finite domains}
\subjclass[2020]{13A15; 13A18; 13F05; 13F20; 13G05; 20K99}

\newcommand{\locpic}{\mathrm{LPic}}
\newcommand{\XX}{\mathbf{X}}
\newcommand{\unit}{\mathcal{U}}
\newcommand{\nz}{\bullet}
\newcommand{\insprinc}{\mathrm{Princ}}
\newcommand{\Int}{\mathrm{Int}}
\newcommand{\Inv}{\mathrm{Inv}}
\newcommand{\Div}{\mathrm{Div}}
\newcommand{\Pic}{\mathrm{Pic}}
\newcommand{\Jac}{\mathrm{Jac}}
\newcommand{\Hom}{\mathrm{Hom}}

\newcommand{\mm}{\mathfrak{m}}
\newcommand{\nn}{\mathfrak{n}}
\newcommand{\coker}{\mathrm{coker}}
\newcommand{\inverse}{\mathrm{inv}}

\newcommand{\njaff}{\mathcal{N}}
\newcommand{\Spechi}{\Spec_{\mathrm{hi}}}
\newcommand{\B}{\mathcal{B}}
\newcommand{\deriv}{\mathcal{D}}

\begin{document}
\begin{abstract}
We study the freeness of the group $\Inv(D)$ of invertible ideals of an integral domain $D$, and the freeness of some related groups of (fractional) ideals. We study the relation between $\Inv(D)$ and $\Inv(D_P)$, in particular in the locally finite case, and we analyze in more detail the case where $D$ is Noetherian (obtaining a characterization of when $\Inv(D)$ is free for one-dimensional analytically unramified Noetherian domains) and where $D$ is Pr\"ufer.
\end{abstract}

\maketitle

\section{Introduction}
Let $D$ be a Dedekind domain. The unique factorization of the ideals of $D$ into product of prime ideals can be expressed by saying that the group $\Inv(D)$ of its invertible ideals is the free abelian group on the set $\Max(D)$ of maximal ideals. Recently, it has been shown that the group $\Inv(D)$ is free also when $D$ is only \emph{locally} Dedekind (or, equivalently, when it is locally a discrete valuation ring), that is, when $D$ is a so-called \emph{almost Dedekind domain}. This has been proved first by considering radical factorization (and the class of SP-domains) \cite{HK-Olb-Re}, and then extending these results with the help of a derived set-like sequence \cite{SP-scattered,bounded-almded}.

In this paper, we start a more general study of condition under which the set $\Inv(D)$ of invertible ideals of $D$ is free; we also extend our reach to closely related groups, such as the group $\insprinc(D)$ of principal ideals, the group $\Inv^t(D)$ of $t$-invertible ideals and the group $\Div(D)$ of $v$-invertible ideals. In Section \ref{sect:locfin} and \ref{sect:preJaff}, we relate the group $\Inv(D)$ with the group of invertible ideals on the localizations of $D$; in particular, we show that if $D$ is a locally finite intersection of domains whose groups of invertible ideals are free then also $\Inv(D)$ is free (Proposition \ref{prop:oplus-tcomplete}), and that the same holds in the case of one-dimensional domains with scattered maximal space (Proposition \ref{prop:preJaff}).

In Section \ref{sect:noeth} we study the case of Noetherian domains: we show that the freeness of $\insprinc(D)$ is related to the group of units of $D$ and of its integral closure $\overline{D}$ (Proposition \ref{prop:ic-krull}) and in the one-dimensional analytically unramified case we characterize when $\Inv(D)$ is free and show that this is a rather strong condition in the non-integrally closed case (Theorem \ref{teor:noeth}). In Section \ref{sect:prufer} we study in more detail the case of Pr\"ufer domains, and we show that under some finiteness hypothesis it is possible to characterize when $\Inv(D)$ and $\Div(D)$ are free (Propositions \ref{prop:cutbranch-semiloc-inv} and \ref{prop:cutbranch-semiloc-div}). We also advance the conjecture that $\Inv(D)$ is free whenever $D$ is a strongly discrete Pr\"ufer domain.

In the final Section \ref{sect:algebras} we relate the group $\Inv(R)$ of invertible ideals of a $D$-algebra $R$ with the group $\Inv(D)$, and prove a sufficient condition for $\Inv(R)$ to be free, based on the properties of the extension $D\subseteq R$ (Theorem \ref{teor:ext-retract}). As a consequence, we show that the group $\Inv(\Int(D))$ (where $\Int(D)$ is the ring of integer-valued polynomials over $D$) is free for every Dedekind domain $D$.

\section{Preliminaries}
Throughout the paper, $D$ is an integral domain, and $K$ will denote its quotient field. We also suppose $D\neq K$, i.e., that $D$ is not a field. We use $\unit(D)$ to denote the group of units of $D$ and $\Jac(D)$ to denote the Jacobson radical of $D$.

\subsection{Fractional and invertible ideals}
A \emph{fractional ideal} of $D$ is a $D$-submodule $I$ of $K$ such that $dI\subseteq D$ for some $d\in K$, $d\neq 0$. We denote by $\insfracid(D)$ the set of fractional ideals of $D$. We shall use the term ``ideal'' to refer to fractional ideals; to refer to ideals in the usual sense (i.e., contained in $D$) we shall use the expression ``integral ideal'' or ``proper ideal''.

A fractional ideal $I$ is \emph{invertible} if there is a fractional ideal $J$ such that $IJ=D$; in this case, $J=(D:I):=\{x\in K\mid xI\subseteq J\}$. Every invertible ideal is finitely generated; moreover, $I$ is invertible if and only if $I$ is finitely generated and locally principal (i.e., $ID_M$ is principal for every $M\in\Max(D)$). In particular, if $D$ is local or semilocal (i.e., if $\Max(D)$ is finite) then every invertible ideal is principal.

The set $\Inv(D)$ of invertible ideals is a group under the product of ideals; $\Inv(D)$ contains as a subgroup the set $\insprinc(D)$ of nonzero principal ideals of $D$. The quotient $\Inv(D)/\insprinc(D)$ is called the \emph{Picard group} of $D$, and is denoted by $\Pic(D)$. The natural map $\unit(K)\longrightarrow\insprinc(D)$, $x\mapsto xD$, induces an exact sequence
\begin{equation*}
0\longrightarrow\unit(D)\longrightarrow\unit(K)\longrightarrow\insprinc(D)\longrightarrow 0.
\end{equation*}


\subsection{Star operations}
A \emph{star operation} is a map $\star:\insfracid(D)\longrightarrow\insfracid(D)$, $I\mapsto I^\star$, such that the following conditions hold for every $I,J\in\insfracid(D)$ and every $x\in K$:
\begin{itemize}
\item $I\subseteq I^\star$;
\item if $I\subseteq J$, then $I^\star\subseteq J^\star$;
\item $(I^\star)^\star=I^\star$;
\item $D^\star=D$;
\item $x\cdot I^\star=(xI)^\star$.
\end{itemize}
The ideal $I^\star$ is said to be the \emph{$\star$-closure} of $I$, and $I$ is said to be \emph{$\star$-closed} (or a \emph{$\star$-ideal}) if $I=I^\star$. A star operation $\star$ is \emph{of finite type} if $I^\star=\bigcup\{J^\star\mid J\subseteq I, J$ finitely generated$\}$.

A \emph{$\star$-maximal} ideal is an ideal that is maximal in the set of the $\star$-ideals contained in $D$; we denote their set with $\Max^\star(D)$. Every $\star$-maximal ideal is prime. If $\star$ is of finite type, then every proper $\star$-ideal is contained in a $\star$-maximal ideal.

The set of star operations has an order, where $\star_1\leq\star_2$ if and only if $I^{\star_1}\subseteq I^{\star_2}$ for every ideal $I$. Under this order, the smallest star operation is the identity (usually denoted by $d$), while the largest is the \emph{divisorial closure} $v$, where $I^v:=(D:(D:I))$; a $v$-closed ideal is called a \emph{divisorial ideal}. The \emph{$t$-operation} is the star operation defined by $I^t:=\{x\in K\mid x\in J^v$ for some finitely generated $J\subseteq I\}$, and is the largest star operation of finite type.

An ideal $I$ is \emph{$\star$-invertible} if there is a $J$ such that $(IJ)^\star=D$; every $\star$-invertible $\star$-ideal is divisorial, and thus is closed by every star operation. The set $\Inv^\star(D)$ of $\star$-closed $\star$-ideals is a group under the ``$\star$-product'' $I\times_\star J:=(IJ)^\star$. If $I_1,\ldots,I_n\in\Inv^\star(D)$, we have $I_1\times_\star\cdots\times_\star I_n=(I_1\cdots I_n)^\star$ \cite[Proposition 32.2(c)]{gilmer}. 

By definition, $\Inv(D)=\Inv^d(D)$; we set $\Div(D):=\Inv^v(D)$. If $\star_1\leq\star_2$, then every $\star_1$-invertible ideal is $\star_2$-invertible; hence, there is a containment $\Inv^{\star_1}(D)\subseteq\Inv^{\star_2}(D)$. In particular, $\Inv(D)$ is contained in $\Inv^\star(D)$ for every $\star$, while $\Div(D)$ contains every $\Inv^\star(D)$.

See \cite[\S 32]{gilmer}, \cite{jaffard_systeme} or \cite{zafrullah_tinvt} for general results about star operations and $\star$-invertibility.

\subsection{Jaffard families}
An \emph{overring} of $D$ is a ring contained between $D$ and $K$; we denote by $\Over(D)$ the set of overrings of $D$. A \emph{flat overring} is an overring that is flat as a $D$-module. We say that a family $\Theta$ of overrings is:
\begin{itemize}
\item \emph{complete} if $I=\bigcap\{IT\mid T\in\Theta\}$ for every ideal $I$;
\item \emph{independent} if, for every $T\neq S$ in $\Theta$, there are no nonzero prime ideals $P$ of $D$ such that $PT\neq T$ and $PS\neq S$; if the elements of $\Theta$ are flat over $D$, this is equivalent to asking that $TS=K$ for all $T\neq S$ in $\Theta$ (cfr. \cite[Section 6.2]{fontana_factoring} and \cite[Lemma 3.4 and Definition 3.5]{jaff-derived});
\item \emph{locally finite} if, for every $x\in D$, we have $xT=T$ for all but finitely many $T\in\Theta$.
\end{itemize}
A domain is \emph{locally finite} if the set $\Theta:=\{D_M\mid M\in\Max(D)\}$ is locally finite.

A \emph{Jaffard family} of $D$ is a family $\Theta$ of flat overrings that is complete, independent and locally finite, and such that $K\notin\Theta$ (see \cite[Section 6.3]{fontana_factoring} and \cite{starloc}). An overring $T$ of $D$ is a \emph{Jaffard overring} if it belongs to a Jaffard family of $D$ \cite[Definition 3.7]{jaff-derived}.

If $D$ is one-dimensional, the set $\Theta:=\{D_M\mid M\in\Max(D)\}$ is always complete and independent, and thus it is a Jaffard family if and only if $\Theta$ (or, equivalently, $D$) is locally finite.

If $\star$ is a star operation on $D$ and $T$ is a flat overring, then $\star$ is said to be \emph{extendable} to $T$ if the map
\begin{equation*}
\begin{aligned}
\star_T\colon \insfracid(T) &\longrightarrow \insfracid(T),\\
IT & \longmapsto I^\star T
\end{aligned}
\end{equation*}
is well-defined, that is, if $IT=JT$ implies $I^\star T=J^\star T$.. (Note that, since $T$ is flat, every fractional ideal is extended from $D$.) Every star operation of finite type is extendable to every flat overring. Moreover, if $T$ is a Jaffard overring then every star operation is extendable to $T$, and every star operation on $T$ is the extension of a star operation on $D$ \cite[Theorem 5.4]{starloc}; in particular, the $d$-, $t$- and $v$-operations on $D$ extend, respectively, to the $d$-, $t$- and $v$-operation on $T$ (use \cite[Theorem 5.6]{starloc}).

\subsection{Topologies}
Let $\Spec(D)$ be the prime spectrum of $D$. We denote by $V(I)$ and $D(I)$, respectively, the basic closed and the basic open subset of $\Spec(D)$ in the Zariski topology induced by the ideal $I$. The \emph{inverse topology} on $\Spec(D)$ (and on its subsets) is the topology generated (as a subbasis of open sets) by the $V(I)$, as $I$ ranges among the finitely generated ideals. We denote by $\Spec(D)^\inverse$ this topological space.

The \emph{Zariski topology} on $\Over(D)$ is the topology generated by the sets $\B(x):=\{T\in\Over(D)\mid x\in T\}$, as $x$ ranges in $K$, while the \emph{inverse topology} is generated by the complements of the $\B(x)$, i.e., by $\Over(D)\setminus\B(x)$ as $x$ ranges in $K$.

A point $x$ of a topological space $X$ is said to be \emph{isolated} if $\{x\}$ is an open set, and it is said to be a \emph{limit point} if it is not isolated. The set of all limit points of $X$ is the \emph{derived set} of $X$, and is denoted by $\deriv(X)$. More generally, for every ordinal number $\alpha$, the \emph{$\alpha$-th derived set} is
\begin{equation*}
\deriv^\alpha(X):=\begin{cases}
X & \text{if~}\alpha=0;\\
\deriv(\deriv^\beta(X)) & \text{if~}\alpha=\beta+1;\\
\bigcap_{\beta<\alpha}\deriv^\beta(X) & \text{if~}\alpha\text{~is a limit ordinal}.
\end{cases}
\end{equation*}
The space $X$ is said to be \emph{scattered} if $\deriv^\alpha(X)=\emptyset$ for some $\alpha$.

\subsection{Valuation and Pr\"ufer domains}
A \emph{valuation domain} is a domain whose ideals are linearly ordered; equivalently, it is a domain $V$ with quotient field $K$ such that there is a surjective map $v:K\setminus\{0\}\longrightarrow\Gamma$ (where $\Gamma$ is a linearly ordered group) such that $v(ab)=v(a)+v(b)$ and $v(a+b)\geq\min\{v(a),v(b)\}$ for every nonzero $a,b\in K$. The group $\Gamma$ is called the \emph{value group} of $V$; we denote it by $\Gamma(V)$. If $\Gamma(V)\simeq\insZ$, then $V$ is said to be a \emph{discrete valuation domain} (DVR).

A prime ideal $P$ of $V$ is \emph{branched} if there is a prime ideal $Q\subsetneq P$ such that there are no prime ideals properly contained between $Q$ and $P$; otherwise $P$ is said to be \emph{unbranched}. Equivalently, $P$ is unbranched if and only if it is the union of the prime ideals properly contained in $P$.

A \emph{Pr\"ufer domain} is an integral domain that is locally a valuation domain, or equivalently such that every finitely generated ideal is invertible. A Pr\"ufer domain $D$ is \emph{strongly discrete} if no prime ideal of $D$ is idempotent. See e.g. \cite{gilmer} or \cite{fontana_libro} for properties of valuation and Pr\"ufer domains.

\subsection{Free groups and homology}
Let $A$ be an abelian group. Then, $A$ is a \emph{free abelian group} if it is the direct sum of infinite cyclic groups; equivalently, $A$ is free if and only if it has a basis, i.e., if there is a subset $S\subset A$ such that every element of $A$ can be written uniquely as a product of elements of $A$. Equivalently, a free abelian group is a free $\insZ$-module. Throughout the paper, whenever we speak of ``free'' groups we always mean free \emph{abelian} groups. 

A free group is always torsionfree, and a subgroup of a free group is always free. A free group is \emph{projective}, i.e., every exact sequence
\begin{equation*}
0\longrightarrow B\longrightarrow C\longrightarrow A\longrightarrow 0
\end{equation*}
splits, and in particular $C\simeq B\oplus A$.

We shall sometimes use the so-called snake lemma: given a commutative diagram with exact rows
\begin{equation*}
\begin{tikzcd}
0\arrow[r] & H_1\arrow[r]\arrow[d,"f"] & G_1\arrow[r]\arrow[d,"g"] & L_1\arrow[r]\arrow[d,"h"] &  0\\
0\arrow[r] & H_2\arrow[r] & G_2\arrow[r] & L_2\arrow[r] & 0.
\end{tikzcd}
\end{equation*}
there is an exact sequence
\begin{equation*}
0\longrightarrow\ker(f)\longrightarrow\ker(g)\longrightarrow\ker(h)\longrightarrow\coker(f)\longrightarrow\coker(g) \longrightarrow\coker(h)\longrightarrow 0.
\end{equation*}

\section{Local finiteness}\label{sect:locfin}
One of the main tools of the paper is the possibility to extend invertible ideals If $D\subseteq R$ is an extension of integral domains, we have a natural map
\begin{equation*}
\begin{aligned}
\phi\colon\Inv(D)  & \longrightarrow \Inv(R),\\
I & \longmapsto IR.
\end{aligned}
\end{equation*}
Indeed, if $I$ is an invertible ideal, then $IJ=D$ for some fractional ideal $J$, and
\begin{equation*}
\phi(I)\phi(J)=IRJR=IJR=DR=R,
\end{equation*}
so that $\phi(I)$ is an invertible ideal of $R$. It is straightforward to see that $\phi$ is a group homomorphism, and that it restricts to a map $\insprinc(D)\longrightarrow\insprinc(R)$. In general, we cannot say much more about $\phi$, as the next examples show.
\begin{ex}
~\begin{enumerate}[(a)]
\item If $R=D[X]$ is the polynomial ring over $D$, then $\phi$ is injective but not surjective; for example, the principal ideal $fD[X]$ is not in the image of $\phi$ for every non-constant polynomial $f$.
\item If $R=D_P$ is the localization of $R$ at a prime $P$, then $\phi$ is not injective, as any principal ideal $xD$ with $x\in D\setminus P$ extends to the whole $R$. However, it is surjective: indeed, every invertible ideal of $D_P$ is principal (since $D_P$ is local) and thus it is an extension of a principal ideal of $D$.
\item If $D$ is Pr\"ufer and $R\in\Over(D)$, then $\phi$ is surjective: indeed, given $I\in\Inv(R)$, let $I_0$ be an ideal of $D$ generated by a finite generating set of $I$: then, $I_0$ is invertible (since it is a finitely generated ideal of a Pr\"ufer domain) and its extension is $I$ by construction.
\item In general, $\phi$ may not be surjective even if $R$ is a localization of $D$. For example, let $D$ be a local Krull domain of dimension $2$, and suppose that there is a non-unit $f\in D$ such that $R:=D[1/f]$ is not a unique factorization domain. Then, $R$ is a Dedekind domain and so $\Inv(R)$ contains non-principal ideals, but every invertible ideal of $D$ is principal (since $D$ is local). Hence there are invertible ideals of $R$ that are not extensions of invertible ideals of $D$.
\end{enumerate}
\end{ex}

More generally, if $\star$ is a star operation on $D$ and $T$ is a flat overring of $D$ such that $\star$ is extendable to $T$, then we have a map
\begin{equation*}
\begin{aligned}
\phi_\star\colon \Inv^\star(D) & \longrightarrow \Inv^{\star_T}(T),\\
I & \longmapsto IT,
\end{aligned}
\end{equation*}
that is a group homomorphism \cite[proof of Proposition 7.1]{starloc}. In particular, this map exists when $\star$ is of finite type (and thus for the $t$-operation), but not in general: for example, it cannot be always extended to the $v$-operation.
\begin{ex}
Let $D$ be a one-dimensional Pr\"ufer domain such that every maximal ideal of $D$ is principal except for one, say $Q$; suppose also that $Q$ is not the radical of any principal ideal and that $D_Q$ is not discrete, and that the Jacobson radical $J$ of $D$ is nonzero. Then, $J=\bigcap\{P\mid P\in\Max(D),P\neq Q\}$ is divisorial, since each prime ideal $P\neq Q$ is finitely generated and thus invertible (and so divisorial). Moreover, since $D$ is completely integrally closed, every ideal is $v$-invertible \cite[Theorem 3.4]{gilmer}, and in particular $J$ is $v$-invertible.

The extension $JD_Q$ is equal to $QD_Q$ since $J$ is radical; however, $QD_Q$ is not divisorial in $D_Q$ since $QD_Q$ is not principal \cite[\S 34, Exercise 12]{gilmer}. Hence, the extension map $\Div(D)\longrightarrow\Div(D_Q)$ is not well-defined.
\end{ex}

The map $\phi$ becomes better-behaved when dealing with Jaffard families.
\begin{prop}\label{prop:jaffard}
Let $D$ be an integral domain, $\star$ a star operation on $D$ and $\Theta$ be a Jaffard family on $D$. Then, there is a group isomorphism
\begin{equation*}
\begin{aligned}
\Phi\colon \Inv^\star(D) & \longrightarrow \bigoplus_{T\in\Theta}\Inv^{\star_T}(T),\\
I & \longmapsto (IT)_{T\in\Theta}.
\end{aligned}
\end{equation*}
In particular,
\begin{enumerate}[(a)]
\item $\displaystyle{\Inv(D)\simeq\bigoplus_{T\in\Theta}\Inv(T)}$;
\item $\displaystyle{\Inv^t(D)\simeq\bigoplus_{T\in\Theta}\Inv^t(T)}$;
\item $\displaystyle{\Div(D)\simeq\bigoplus_{T\in\Theta}\Div(T)}$.
\end{enumerate}
Thus $\Inv(D)$ (respectively, $\Inv^t(D)$, $\Div(D)$) is free if and only if $\Inv(T)$ (respectively, $\Inv^t(T)$, $\Div(T)$) is free for every $T\in\Theta$.
\end{prop}
\begin{proof}
The first part of the theorem is exactly \cite[Proposition 7.1]{starloc}, and the second part follows by taking $\star=d$, $\star=t$ and $\star=v$, respectively. The last part follows trivially.
\end{proof}

\begin{cor}
Let $D$ be a locally finite one-dimensional domain. Then, $\displaystyle{\Inv(D)\simeq\bigoplus\{\Inv(D_M)\mid M\in\Max(D)\}}$; in particular, $\Inv(D)$ is free if and only if $\Inv(D_M)$ is free for every maximal ideal $M$.
\end{cor}
\begin{proof}
It is enough to note that $\Theta:=\{D_M\mid M\in\Max(D)\}$ is a Jaffard family and apply Proposition \ref{prop:jaffard}.
\end{proof}

Usually, we do not have at our disposal a Jaffard family. However, if we are interested only in the freeness of $\Inv(D)$, we can work with much weaker hypothesis.
\begin{lemma}\label{lemma:extInv}
Let $D$ be an integral domain, and let $\Theta$ be a family of overrings of $D$ such that $D=\bigcap\{T\mid T\in\Theta\}$. Then, there is an injective group homomorphism
\begin{equation*}
\begin{aligned}
\Phi\colon \Inv(D)  & \longrightarrow \prod_{T\in\Theta}\Inv(T).\\
I & \longmapsto (IT)_{T\in\Theta}.
\end{aligned}
\end{equation*}
Moreover, if each $T\in\Theta$ is flat, then there is an injective group homomorphism
\begin{equation*}
\begin{aligned}
\Phi_t\colon \Inv^t(D)  & \longrightarrow \prod_{T\in\Theta}\Inv^t(T).\\
I & \longmapsto (IT)_{T\in\Theta}.
\end{aligned}
\end{equation*}
\end{lemma}
\begin{proof}
The discussion in the previous part of the section shows that $\Phi$ is a well-defined group homomorphism, since each component $\Inv(D)\longrightarrow\Inv(T)$ is a homomorphism. Likewise, we have a homomorphism $\Inv^t(D)\longrightarrow\Inv^{t_T}(D)$ (where $t_T$ is the extension of the $t$-operation), and $\Inv^{t_T}(D)$ is contained in $\Inv^t(T)$ since $t_T$ is of finite type and the $t$-operation is the largest star operation of finite type.

To show that these maps are injective, we note that the map $\wedge_\Theta:J\mapsto\bigcap\{JT\mid T\in\Theta\}$ is a star operation on $D$; since every $t$-invertible ideal (and, in particular, every invertible ideal) is divisorial, it is also $\wedge_\Theta$-closed, and thus for every $I\in\Inv^t(D)$ we have $I=\bigcap\{IT\mid T\in\Theta\}$, from which the injectivity of $\Phi$ and $\Phi_t$ follows. The claim is proved.
\end{proof}

\begin{prop}\label{prop:oplus-tcomplete}
Let $D$ be an integral domain. Let $\Theta$ be a family of overrings of $D$ such that:
\begin{itemize}
\item $\Theta$ is locally finite;
\item $D=\bigcap\{T\mid T\in\Theta\}$
\item $\Inv(T)$ is free for every $T\in\Theta$.
\end{itemize}
Then, $\Inv(D)$ is free.
\end{prop}
\begin{proof}
Consider the map $\Phi$ of Lemma \ref{lemma:extInv}. Since $\Theta$ is locally finite, given an $I\in\Inv(D)$ we have $IT=T$ for all but finitely many $T\in\Theta$; thus the range of $\Phi$ is contained in the direct sum $\bigoplus\Inv(T)$, which is free since each $\Inv(T)$ is free. Hence $\Inv(D)\simeq\Phi(\Inv(D))$ is (isomorphic to) a subgroup of a free group, and thus it is itself free. The claim is proved.
\end{proof}

\begin{cor}\label{cor:locfin}
Let $D$ be an integral domain. If $D$ is locally finite and $\Inv(D_M)$ is free for every $M\in\Max(D)$, then $\Inv(D)$ is free.
\end{cor}
\begin{proof}
Saying that $D$ is locally finite is equivalent to saying that $\Theta:=\{D_M\mid M\in\Max(D)\}$ is locally finite. The claim now follows from Proposition \ref{prop:oplus-tcomplete}.
\end{proof}

We can also obtain a similar result for $t$-invertible ideals. Recall that a domain has the \emph{$t$-finite character} if the set $\{D_M\mid M\in\Max^t(D)\}$ is locally finite.
\begin{prop}\label{prop:oplus-tmax}
Let $D$ be a domain with the $t$-finite character. If $\Inv^t(D_M)$ is free for every $M\in\Max^t(D)$, then $\Inv^t(D)$ is free.
\end{prop}
\begin{proof}
It is enough to apply Lemma \ref{lemma:extInv} on $\Theta:=\{D_M\mid M\in\Max^t(D)\}$ and use the same method of the proof  of Proposition \ref{prop:oplus-tcomplete} with the $t$-finite character to guarantee the local finiteness of $\Theta$.
\end{proof}

\section{The Noetherian case}\label{sect:noeth}
In this section, we study the case in which $D$ is a Noetherian domain.

The integrally closed case can be done in greater generality; recall that $D$ is a \emph{Krull domain} is an integral domain that is the intersection of a locally finite family of discrete valuation rings, each of which is a localization of $D$. Every Noetherian integrally closed domain is a Krull domain and, indeed, the integral closure of every Noetherian domain (in its quotient field) is a Krull domain \cite{nagata-morinagata}.
\begin{prop}\label{prop:krull}
Let $D$ be a Krull domain. Then, $\Div(D)$, $\Inv(D)$ and $\insprinc(D)$ are free groups.
\end{prop}
\begin{proof}
Since $D$ is a Krull domain, the $v$-operation coincides with the $t$-operation and thus $\Div(D)=\Inv^t(D)$. By Proposition \ref{prop:oplus-tmax}, $\Inv^t(D)$ is free; as a consequence, also its subgroups $\Inv(D)$ and $\insprinc(D)$ are free.
\end{proof}

We note that, for Krull domains, Lemma \ref{lemma:extInv} and Proposition \ref{prop:oplus-tmax} give a map
\begin{equation*}
\begin{aligned}
\widetilde{\Phi}\colon\Inv^t(D)  & \longrightarrow\bigoplus_{M\in\Max^t(D)}\Inv^t(D_M),\\
I & \longmapsto (ID_M)_{M\in\Max^t(D)}
\end{aligned}
\end{equation*}
that becomes an isomorphism, since $D_M$ is a DVR whenever $M\in\Max^t(D)$ and thus the $t$-operation on $D_M$ is the identity. Therefore, we obtain back the well-known result (see e.g. \cite[Chapter 1, Corollary 3.4]{fossum-krull}) that, for a Krull domain, the set $X^1(D)$ of prime ideals of height $1$ is a basis of $\Inv^t(D)=\Div(D)$.

In the non-integrally closed case, the freeness of $\insprinc(D)$ and $\Inv(D)$ is strongly connected to the units of $D$.
\begin{prop}\label{prop:ic-krull}
Let $D$ be an integral domain such that its integral closure $\overline{D}$ is a Krull domain. Then, $\insprinc(D)$ is free if and only if $\unit(\overline{D})/\unit(D)$ is free.
\end{prop}
\begin{proof}
The natural maps $\unit(D)\longrightarrow\unit(\overline{D})$ and $\insprinc(D)\longrightarrow\insprinc(\overline{D})$, $xD\mapsto x\overline{D}$, give rise to a commutative diagram
\begin{equation*}
\begin{tikzcd}
0\arrow[r] & \unit(D)\arrow[r]\arrow[d] & \unit(K)\arrow[r]\arrow[d,equal] & \insprinc(D)\arrow[r]\arrow[d] &  0\\
0\arrow[r] & \unit(\overline{D})\arrow[r] & \unit(K)\arrow[r] & \insprinc(\overline{D})\arrow[r] & 0.
\end{tikzcd}
\end{equation*}
The leftmost vertical map is injective, the middle one is the identity and the rightmost one is surjective; by the snake lemma, the kernel of $\insprinc(D)\longrightarrow\insprinc(\overline{D})$ is isomorphic to the cokernel of $\unit(D)\longrightarrow\unit(\overline{D})$, i.e., there is an exact sequence
\begin{equation*}
0\longrightarrow \frac{\unit(\overline{D})}{\unit(D)}\longrightarrow \insprinc(D)\longrightarrow\insprinc(\overline{D})\longrightarrow 0.
\end{equation*}
Since $\overline{D}$ is Krull, the group $\insprinc(\overline{D})$ is free by Proposition \ref{prop:krull}; thus, the sequence splits and $\insprinc(D)\simeq\insprinc(\overline{D})\oplus\frac{\unit(\overline{D})}{\unit(D)}$. The claim follows.
\end{proof}

There are two immediate specializations of the previous proposition that are of interest.
\begin{cor}
Let $D$ be an local integral domain such that its integral closure $\overline{D}$ is a Krull domain. Then, $\Inv(D)$ is free if and only if $\unit(\overline{D})/\unit(D)$ is free.
\end{cor}
\begin{proof}
If $D$ is free, then $\Inv(D)=\insprinc(D)$. The claim follows from Proposition \ref{prop:ic-krull}.
\end{proof}

\begin{cor}\label{cor:noeth-quozunit}
Let $D$ be a Noetherian domain and $\overline{D}$ its integral closure. Then, $\insprinc(D)$ is free if and only if $\unit(\overline{D})/\unit(D)$ is free.
\end{cor}
\begin{proof}
If $D$ is Noetherian, then $\overline{D}$ is a Krull domain \cite{nagata-morinagata}. The claim follows from Proposition \ref{prop:ic-krull}.
\end{proof}

Putting the previous corollaries together, we have:
\begin{cor}\label{cor:noethloc-quozunit}
Let $D$ be a Noetherian local domain and $\overline{D}$ its integral closure. Then, $\Inv(D)$ is free if and only if $\unit(\overline{D})/\unit(D)$ is free.
\end{cor}

We now study in more detail the one-dimensional case.

Let $\mathfrak{c}:=(D:\overline{D})$ be the conductor of the extension $D\subseteq\overline{D}$. Then, $\mathfrak{c}\neq(0)$ if and only if $D$ is analytically unramified, i.e., if and only if the completion of $D$ is reduced. In this case, we can study the quotient $\unit(\overline{D})/\unit(D)$ through quotienting by $\mathfrak{c}$.
\begin{lemma}\label{lemma:quozc-surj}
Let $A$ be a domain, and $I\subseteq\Jac(A)$ be an ideal. Then, the natural map $\phi:\unit(A)\longrightarrow\unit(A/I)$ is surjective.
\end{lemma}
\begin{proof}
Let $u=a+I\in\unit(A/I)$. Then, there is a $b\in A$ such that $(a+I)(b+I)=1+I$, i.e., $ab-1\in I\subseteq\Jac(A)$. Hence, $ab\in 1+\Jac(A)\subseteq\unit(A)$. Thus $a\in\unit(A)$ and $u=\phi(a)$, so that $\phi$ is surjective.
\end{proof}

\begin{lemma}\label{lemma:ext->surj}
Let $A\subseteq B$ be an extension of domains, and let $I$ be a common ideal of $A$ and $B$ that is contained in their Jacobson radical. Then, $\displaystyle{\frac{\unit(B)}{\unit(A)}\simeq\frac{\unit(B/I)}{\unit(A/I)}}$.
\end{lemma}
\begin{proof}
Let $\phi_A:\unit(A)\longrightarrow\unit(A/I)$ and $\phi_B:\unit(B)\longrightarrow\unit(B/I)$ be the natural maps; by Lemma \ref{lemma:quozc-surj}, they are surjective. Consider the commutative diagram
\begin{equation*}
\begin{tikzcd}
0\arrow[r] & \unit(A)\arrow[d,"\phi_A"]\arrow[r] & \unit(B)\arrow[d,"\phi_B"]\arrow[r] & \unit(B)/\unit(A)\arrow[d,"\psi"]\arrow[r] & 0\\
0\arrow[r] & \unit(A/I)\arrow[r] & \unit(B/I)\arrow[r] & \unit(B/I)/\unit(A/I)\arrow[r] & 0.
\end{tikzcd}
\end{equation*}
By the snake lemma, the surjectivity of $\phi_A$ and $\phi_B$ implies that $\psi$ is surjective, and that we have an exact sequence
\begin{equation*}
0\longrightarrow\ker\phi_A\longrightarrow\ker\phi_B\longrightarrow\ker\psi\longrightarrow 0.
\end{equation*}
The kernels of $\phi_A$ and $\phi_B$ coincide, since they are both equal to $1+I$; thus, we must also have $\ker\psi=0$. Hence, $\psi$ is both injective and surjective, thus a homeomorphism.
\end{proof}

\begin{lemma}\label{lemma:additive-omf-free}
Let $K$ be a field and $F$ be a free group; let $G$ be the additive group of $K$. Then, for every integer $n$, $\Hom(G^n,F)=0$. In particular, no quotient of $G^n$ is free.
\end{lemma}
\begin{proof}
As subgroups of free groups are free, it is enough to consider surjective homomorphisms; moreover, composing with a nonzero homomorphism $F\longrightarrow\insZ$, we can suppose that $F=\insZ$ is cyclic. Furthermore, since $G^n$ is the direct sum of $n$ copies of $G$, by the properties of the hom-sets we can suppose that $n=1$.

Let thus $\phi:G\longrightarrow\insZ$ be such a homomorphism. 

Let $n>1$ be a positive integer that is not a multiple of to the characteristic of $K$. Since $\phi$ is surjective, there is an $x$ such that $\phi(x)=1$. By the choice of $n$, there is an $y$ such that $x=ny$: thus $1=\phi(x)=\phi(ny)=n\phi(y)$, which is impossible in $\insZ$. Thus $\phi$ must be the zero homomorphism, as claimed.

The ``in particular'' statement follows from the fact that a quotient $G^n\longrightarrow H$ is a nonzero homomorphism.
\end{proof}

We now distinguish three cases according to the properties of $\mathfrak{c}$ as an ideal of $\overline{D}$.
\begin{prop}\label{prop:noeth-norad}
Let $(D,\mm)$ be a local Noetherian domain of dimension $1$, and let $\overline{D}$ be its integral closure; suppose that $D\neq\overline{D}$ and that $\mathfrak{c}:=(D:\overline{D})\neq(0)$. If $\mathfrak{c}$ is not a radical ideal of $\overline{D}$, then $\Inv(D)$ is not free.
\end{prop}
\begin{proof}
Let $\mm_1,\ldots,\mm_n$ be the maximal ideals of $\overline{D}$ (there are only finitely many of them since $\overline{D}$ is the integral closure of the one-dimensional local Noetherian domain); note that $\mathfrak{c}\subseteq\mm_i$ for every $i$ and thus $\mathfrak{c}\subseteq\Jac(\overline{D})$. Since $\mathfrak{c}\neq(0)$, we can write $\mathfrak{c}=\mm_1^{e_1}\cdots\mm_n^{e_n}$ for some natural numbers $e_i>0$. As $\mathfrak{c}$ is not radical, we have $e_i>1$ for some $i$; without loss of generality we can suppose that $e_1>1$.

Let $A:=D/\mathfrak{c}$ and $B:=\overline{D}/\mathfrak{c}$: then, $A$ is a local Artinian ring while $B$ is an Artinian ring with $n$ maximal ideals, say $\nn_i:=\mm_i/\mathfrak{c}$. By Lemma \ref{lemma:ext->surj}, $\unit(B)/\unit(A)\simeq\unit(\overline{D})/\unit(D)$; by Corollary \ref{cor:noethloc-quozunit}, we only need to prove that $\unit(B)/\unit(A)$ is not free.

Let $I:=\nn_1^{e_1-1}\cdots\nn_n^{e_n}$; then, $I\subseteq\Jac(B)$ and $(0)=I\nn_1$. Moreover, $I\nsubseteq A$ since $I$ is the image of $I_0:=\mm_1^{e_1-1}\cdots\mm_n^{e_n}$ and $I_0$ is larger than $\mathfrak{c}=(D:\overline{D})$, that by definition is the largest ideal that is common to $D$ and $\overline{D}$.

Consider $H:=1+I$. Since $I\subseteq\Jac(B)$, $H\subseteq\unit(B)$; moreover, if $1+t,1+t'\in H$, then
\begin{equation*}
(1+t)(1+t')=1+t+t'+tt'=1+t+t'
\end{equation*}
since $tt'\in I^2=\nn_1^{2(e_1-1)}\cdots\nn_n^{2e_n}=(0)$ (using $e_1>1$). Therefore, $H$ is a subgroup of $\unit(B)$, and the map
\begin{equation*}
\begin{aligned}
H  & \longrightarrow (I,+),\\
1+t & \longmapsto t
\end{aligned}
\end{equation*}
is a group isomorphism. By construction, $I\nn_1=(0)$; therefore, $I$ is a $k:=B/\nn_1$-vector space, that is of finite dimension since $B$ is Noetherian.

The quotient $H':=H/(H\cap\unit(A))$ is a subgroup of $\unit(B)/\unit(A)$; if the latter were free, then also $H'$ would be free. Moreover, since $I\nsubseteq A$ we have $H'\neq(0)$. However, $H'$ is isomorphic to a quotient of $(I,+)$, which is isomorphic to the direct sum of finitely many copies of $(k,+)$; by Lemma \ref{lemma:additive-omf-free}, no such quotient can be free, and thus $H'$ is not free. Hence $\unit(B)/\unit(A)$ is not free, and thus neither are $\unit(\overline{D})/\unit(D)$ and $\Inv(D)$.
\end{proof}

We now turn to the case where $\mathfrak{c}$ is a nonzero radical ideal. One case can only be solved almost tautologically.
\begin{prop}\label{prop:noeth-rad-ovDloc}
Let $(D,\mm)$ be a local Noetherian domain of dimension $1$, and let $\overline{D}$ be its integral closure; suppose that $\mathfrak{c}:=(D:\overline{D})$ is the unique maximal ideal of $\overline{D}$. Then, $\Inv(D)$ is free if and only if $\unit(\overline{D}/\mathfrak{c})/\unit(D/\mm)$ is free.
\end{prop}
\begin{proof}
By Lemma \ref{lemma:ext->surj} and Corollary \ref{cor:noeth-quozunit}, $\Inv(D)$ is free if and only if $\unit(\overline{D})/\unit(D)\simeq\unit(\overline{D}/\mathfrak{c})/\unit(D/\mm)$ is free.
\end{proof}

When $\overline{D}$ has more than one maximal ideal, the condition equivalent to the freeness of $\Inv(D)$ is more complicated. We premise a group-theoretic lemma.
\begin{lemma}\label{lemma:amalgamato}
Let $G,A_1,\ldots,A_n$ be group. For each $i$, let $\phi_i:G\longrightarrow A_i$ be injective homomorphisms and suppose that $A_i=B_i\oplus\phi_i(G)$ for some subgroup $B_i$. Let $\phi:G\longrightarrow A_1\oplus\cdots\oplus A_n$ be the compositum of all the $\phi_i$. Then, $(A_1\oplus\cdots\oplus A_n)/\phi(G)\simeq B_1\oplus\cdots\oplus B_n\oplus G^{n-1}$.
\end{lemma}
\begin{proof}
For each $i$, let $\pi_i:A_i\longrightarrow B_i$ and be the canonical quotient, and let $\theta_i:A_i\longrightarrow G_i$ be the left inverse of $\phi_i$. Consider the following two maps:
\begin{equation*}
\begin{aligned}
\psi_1\colon A_1\oplus\cdots\oplus A_n &\longrightarrow B_1\oplus\cdots\oplus B_n,\\
(a_1,\ldots,a_n) & \longmapsto (\pi_1(a_1),\ldots,\pi_n(a_n)),
\end{aligned}
\end{equation*}
and
\begin{equation*}
\begin{aligned}
\psi_2\colon A_1\oplus\cdots\oplus A_n &\longrightarrow G^{n-1},\\
(a_1,\ldots,a_n) & \longmapsto (\theta_1(a_1)-\theta_n(a_n),\theta_2(a_1)-\theta_n(a_n),\ldots,\theta_{n-1}(a_{n-1})-\theta_n(a_n)).
\end{aligned}
\end{equation*}
Let $\psi:A_1\oplus\cdots\oplus A_n\longrightarrow B_1\oplus\cdots\oplus B_n\oplus G^{n-1}$ be the compositum of $\psi_1$ and $\psi_2$. We claim that $\psi$ is surjective with kernel $\phi(G)$.

Indeed, if $(a_1,\ldots,a_n)\in\ker\psi$ then $a_i\in\phi_i(G)$, and thus $a_i=\phi_i(b_i)$ for some $b_i\in G$; by definition, $\theta_i(a_i)=(\theta_i\circ\phi_i)(g_i)=g_i$, and thus the condition $\theta_i(a_i)=\theta_n(a_n)$ implies $b_i=b_n$ for each $n$, that is, all the $g_i$ are equal to the same $g$. Therefore, $a_i=\phi_i(g)$, and $(a_1,\ldots,a_n)=\phi(g)\in\phi(G)$. The inclusion $\phi(G)\subseteq\ker\psi$ follows trivially.

Let now $(b_1,\ldots,b_n,g_1,\ldots,g_{n-1})\in B_1\oplus\cdots\oplus B_n\oplus G^{n-1}$, and set $a_i=(b_i,g_i)$ for $i<n$ and $a_n=(b_n,0)$. Then, $b_i=\pi_i(a_i)$ for each $i$, while $\theta_i(a_i)-\theta_n(a_n)=g_i$; therefore, $\psi(a_1,\ldots,a_n)=(b_1,\ldots,b_n,g_1,\ldots,g_{n-1})$. Thus $\psi$ is surjective, and the claim is proved.
\end{proof}

\begin{prop}\label{prop:noeth-rad-ovDnonloc}
Let $(D,\mm)$ be a local Noetherian domain of dimension $1$, and let $\overline{D}$ be its integral closure; suppose that $\overline{D}$ is not local. Let $\mathfrak{c}:=(D:\overline{D})$; suppose that $\mathfrak{c}\neq(0)$ is a radical ideal of $\overline{D}$. Let $k$ be the residue field of $D$ and $L_1,\ldots,L_n$ be the residue field of $\overline{D}$. Then, $\Inv(D)$ is free if and only if, for each $i$, $\unit(L_i)$ is free and the natural map $\phi_i:\unit(k)\longrightarrow\unit(L_i)$ induced by the inclusion $k\hookrightarrow L_i$ makes $\phi_i(\unit(k))$ a direct summand of $\unit(L_i)$.
\end{prop}
\begin{proof}
By Lemma \ref{lemma:ext->surj} and Corollary \ref{cor:noeth-quozunit}, we need to show that $\unit(B)/\unit(A)$ is free, where $A:=D/\mathfrak{c}$ and $B:=\overline{D}/\mathfrak{c}$. By hypothesis, $B\simeq L_1\oplus\cdots\oplus L_n$, and thus $\unit(B)\simeq\unit(L_1)\oplus\cdots\oplus\unit(L_n)$; on the other hand, $\unit(A)\simeq\unit(k)$, and the inclusion $\unit(A)\hookrightarrow\unit(B)$ is the diagonal embedding of $\unit(k)$ into the product. Note that $n>1$ since $\overline{D}$ is not local.

Suppose first that the conditions in the statement hold. Then, we are in the setting of Lemma \ref{lemma:amalgamato}: $\unit(L_i)\simeq B_i\oplus\phi_i(\unit(k))$ and thus $\unit(B)/\unit(A)$ is isomorphic to $B_1\oplus\cdots\oplus B_n\oplus\unit(k)^{n-1}$. Since each $\unit(L_i)$ is free, so are the $B_i$ and $\unit(k)$ (as they are isomorphic to subgroups of a free group), and thus $\unit(B)/\unit(A)$ is free, as claimed.

Conversely, suppose that $\unit(B)/\unit(A)$ is free. Consider the subgroup $H_1:=\unit(L_1)\oplus(1,\ldots,1)$: then, $\unit(A)\cap H_1=\phi_1(\unit(k))$, and thus $H_1/(\unit(A)\cap H_1)\simeq\unit(L_1)/\phi_i(\unit(k))$. Since $H_1/(\unit(A)\cap H_1)$ is a subgroup of $\unit(B)/\unit(A)$, it must be free; hence, the exact sequence
\begin{equation*}
0\longrightarrow\unit(k)\longrightarrow\unit(L_1)\longrightarrow\unit(L_1)/\phi_1(\unit(k))\longrightarrow 0
\end{equation*}
splits and so $\unit(L_1)\simeq B_1\oplus\phi_1(\unit(k))$ for some free subgroup $B_1$. The same holds for each $\unit(L_i)$. An application of Lemma \ref{lemma:amalgamato} now implies that $\unit(B)/\unit(A)$ is isomorphic to $B_1\oplus\cdots\oplus B_n\oplus\unit(k)^{n-1}$; therefore, also $\unit(k)$ must be free, and so each $\unit(L_1)$ is free, being the direct sum of two free groups. Therefore, the conditions of the statement are fulfilled.
\end{proof}

In the proposition above, the condition that the unit group of a field is free is very rare.
\begin{lemma}\label{lemma:unit-artin}
Let $K$ be a field. If $\unit(K)$ is free, then $A$ is has characteristic $2$ and $\ins{F}_2$ is algebraically closed in $K$.
\end{lemma}
\begin{proof}
If the characteristic of $K$ is not $2$ then $\unit(K)$ contains the torsion element $-1\neq 1$, and thus it is not free. If $K$ has characteristic $2$, then any $x\in K$ that is algebraic over $\ins{F}_2$ is torsion; thus if $\unit(K)$ is free then $\ins{F}_2$ must be algebraically closed in $K$.
\end{proof}

\begin{cor}
Let $(D,\mm)$ be a local Noetherian domain of dimension $1$, and let $\overline{D}$ be its integral closure; suppose that $\overline{D}$ is not local. Let $\mathfrak{c}:=(D:\overline{D})$; suppose that $\mathfrak{c}\neq(0)$ is a radical ideal of $\overline{D}$. If $\Inv(D)$ is free, then the characteristic of $D/\mm$ is $2$.
\end{cor}
\begin{proof}
If $\Inv(D)$ is free, by Proposition \ref{prop:noeth-rad-ovDnonloc} the multiplicative group of the residue fields of $\overline{D}$ must be free, and thus also $\unit(D/\mm)$ must be free. By Lemma \ref{lemma:unit-artin}, it follows that $D/\mm$ has characteristic $2$.
\end{proof}

We collect the previous results in the following theorem.
\begin{teor}\label{teor:noeth}
Let $(D,\mm)$ be a local Noetherian domain of dimension $1$, and suppose that $D$ is analytically unramified and not integrally closed. Let $\overline{D}$ be its integral closure and $\mathfrak{c}:=(D:\overline{D})$. Suppose that $\mathfrak{c}\neq(0)$. Let $k$ be the residue field of $D$ and $L_1,\ldots,L_n$ be the residue fields of $\overline{D}$, and consider $k$ as a subfield of each $L_i$. Then, the following hold.
\begin{enumerate}[(a)]
\item If $\mathfrak{c}$ is not radical in $\overline{D}$, then $\Inv(D)$ is not free.
\item If $\overline{D}$ is local and $\mathfrak{c}$ is radical in $\overline{D}$, then $\Inv(D)$ is free if and only if $\unit(L_1)/\unit(k)$ is free.
\item If $\overline{D}$ is not local and $\mathfrak{c}$ is radical in $\overline{D}$, then $\Inv(D)$ is free if and only if, for every $i$, $\unit(L_i)$ is free and $\unit(k)$ is a direct summand of $\unit(L_i)$.
\end{enumerate}
\end{teor}
\begin{proof}
Put together Proposition \ref{prop:noeth-norad}, Proposition \ref{prop:noeth-rad-ovDloc} and Proposition \ref{prop:noeth-rad-ovDnonloc}.
\end{proof}

\begin{oss}
~\begin{enumerate}
\item In Proposition \ref{prop:noeth-rad-ovDloc}, it is possible for $\unit(\overline{D})/\unit(D)$ to be free even if $\unit(D)\neq\unit(\overline{D})$. For example, let $K\subseteq L$ be a field extension, and let $D:=K+XL[X]_{(X)}$; i.e., $D$ is the pullback of $K$ inside $\overline{D}=L[X]_{(X)}$. Then, $\unit(D)=\unit(K)+XL[X]_{(X)}$ while $\unit(\overline{D})=\unit(L)+XL[X]_{(X)}$, and thus $\unit(\overline{D})/\unit(D)\simeq \unit(L)/\unit(K)$. Suppose now that $K=\insQ$ and $K\subseteq L$ is finite. By \cite[Proposition 1]{may-multgroup-fieldext} $\unit(L)/\unit(K)\simeq F\oplus A$, where $F$ is free and $A$ is finite. In particular, if $A$ is trivial then $\unit(\overline{D})/\unit(D)$ is free.

For an explicit example, let $\alpha:=\zeta_7+\zeta_7^{-1}$ (where $\zeta_7\neq 1$ and $\zeta_7^7=1$) and $L:=\insQ(\alpha)$: then, $L$ is a Galois extension of $\insQ$ of degree $3$, and if $A$ is not trivial there would be an $x\in L\setminus\insQ$ such that $x^k\in\insQ$ for some $k$. Thus $L$ should contain the Galois closure of $\insQ(\sqrt[3]{y})$, which however has degree $6$ over $\insQ$.

\item Corollary \ref{cor:noethloc-quozunit} does not hold when $D$ is not local. For example, suppose that $K$ be a field and let $D=K[X^2,X^3]$. Then, $\overline{D}=K[X]$ and thus $\unit(D)=K\setminus\{0\}=\unit(\overline{D})$. However, if $\Inv(D)$ is free then so is $\Inv(D_M)$, for every maximal ideal $M$ (by Proposition \ref{prop:jaffard}, since $D$ is a locally finite domain): choosing $M=(X^2,X^3)$, we obtain that $\overline{D_M}=K[X]_{(X)}$ and $(D_M:\overline{D_M})=X^2\overline{D_M}$ is not radical, and thus $\Inv(D_M)$ is not free by Proposition \ref{prop:noeth-norad}. 

\item If $K$ is a field of characteristic $2$, then it is possible for $\unit(K)$ to be free even if $K\neq\ins{F}_2$. For example, if $\XX$ is a set of indeterminates and $K=\ins{F}_2(\XX)$, then $\unit(K)$ is isomorphic to $\insprinc(K[\XX])$, which is free since $K[\XX]$ is a unique factorization domain.

\item Likewise, the fact that $\unit(K)$ and $\unit(K')$ are free (where $K\subseteq K'$ are fields) does not guarantee that $\unit(K)$ is a direct summand of $\unit(K')$. Indeed, suppose that $X$ is an indeterminate over $\ins{F}_2$, and let $K=\ins{F}_2(X^2)$ and $K'=\ins{F}_2(X)$. Then, $\unit(K)\simeq\unit(K')$ are free, but $X+\unit(K')$ has degree $2$ in the quotient, so that $\unit(K')/\unit(K)$ is not torsionfree and, henceforth, not free.
\end{enumerate}
\end{oss}

\section{Beyond local finiteness}\label{sect:preJaff}
In this section, we generalize Proposition \ref{prop:jaffard} from Jaffard to pre-Jaffard families. Recall that a \emph{pre-Jaffard family} on $D$ is a set $\Theta$ of flat overrings of $D$ that is complete, independent and compact, with respect to the Zariski topology of $\Over(D)$, and such that $K\notin\Theta$.

For every ordinal number $\alpha$, we associate to $\Theta$ a subset $\njaff^\alpha(\Theta)$ and an overring $T_\alpha$ of $D$ in the following way:
\begin{itemize}
\item $\njaff^0(\Theta):=\Theta$, $T_0:=D$;
\item if $\alpha=\gamma+1$ is a successor ordinal, then $\njaff^\alpha(\Theta)$ is the set of members of $\njaff^\gamma(\Theta)$ that are not Jaffard overrings of $T_\gamma$;
\item if $\alpha$ is a limit ordinal, then $\njaff^\alpha(\Theta):=\bigcap\{\njaff^\gamma(\Theta)\mid \gamma<\alpha\}$;
\item $T_\alpha:=\bigcap\{T\mid T\in\njaff^\alpha(\Theta)\}$.
\end{itemize}
Then, $\{\njaff^\alpha(\Theta)\}_\alpha$ is a decreasing sequence of subsets of $\Theta$ and $\{T_\alpha\}_\alpha$ is an increasing sequence of overrings of $D$; moreover, $\njaff^\alpha(\Theta)$ is always a pre-Jaffard family of $T_\alpha$. If $T_\alpha=K$ for some $\alpha$, we say that $\Theta$ is \emph{sharp}. See \cite{jaff-derived} for properties of pre-Jaffard families and of this sequence.

\begin{lemma}\label{lemma:explostion-prejaff}
Let $D$ be an integral domain, $\Theta$ a pre-Jaffard family of $D$, and let $I\neq D$ be a finitely generated fractional ideal of $D$ such that $IT_\alpha=T_\alpha$ for some ordinal number $\alpha$. Then, there is an ordinal $\beta$ such that $IT_\beta\neq T_\beta$ and $IT_{\beta+1}=T_{\beta+1}$.
\end{lemma}
\begin{proof}
Let $\Gamma$ be the set of all ordinal numbers $\lambda$ such that $IT_\lambda=T_\lambda$; then, $\alpha\in\Gamma$, and thus $\Gamma$ has a minimum $\gamma$. If $\gamma=\beta+1$ is a successor ordinal, then $\beta$ is the ordinal we were looking for; we claim that $\gamma$ cannot be a limit ordinal.

Indeed, let $I=(x_1,\ldots,x_m)$. Since $IT_\gamma=T_\gamma$, we have $x_1,\ldots,x_m\in T_\gamma$, and there are $t_1,\ldots,t_m\in T_\gamma$ such that $1=x_1t_1+\cdots+x_mt_m$. If $\gamma$ is a limit ordinal, then by \cite[Lemma 7.1]{PicInt} we have $T_\gamma=\bigcup_{\lambda<\gamma}T_\lambda$, and thus there is a $\overline{\lambda}<\gamma$ such that $x_1,\ldots,x_m,t_1,\ldots,t_m\in T_{\overline{\lambda}}$. However, this implies that $IT_{\overline{\lambda}}=T_{\overline{\lambda}}$, against the definition of $\gamma$; hence $\gamma$ cannot be a limit ordinal, as claimed.
\end{proof}

The proof of the following proposition follows the ideas laid out in \cite[Section 5]{SP-scattered} and \cite[Section 7]{PicInt}.
\begin{prop}\label{prop:preJaff}
Let $D$ be an integral domain, $\Theta$ a pre-Jaffard family on $D$, and let $\{T_\alpha\}_\alpha$ be the derived sequence. If $\Inv(T)$ is free for every $T\notin\njaff^\alpha(\Theta)$, then there is an exact sequence
\begin{equation}\label{eq:exseq-Talpha}
0\longrightarrow\bigoplus_{T\in\Theta\setminus\njaff^\alpha(\Theta)}\Inv(T)\longrightarrow\Inv(D)\longrightarrow\Inv(T_\alpha) \longrightarrow 0.
\end{equation}
In particular, if $\Theta$ is sharp with Jaffard degree at most $\alpha$, then $\displaystyle{\Inv(D)\simeq\bigoplus_{T\in\Theta}\Inv(T)}$, and $\Inv(D)$ is free.
\end{prop}
\begin{proof}
By \cite[Theorem 7.2(a)]{PicInt}, the extension map $\Inv(D)\longrightarrow\Inv(T_\alpha)$ is always surjective; we denote by $K_\alpha$ its kernel. We proceed by induction on $\alpha$.

If $\alpha=0$ the claim is obvious. If $\alpha=1$, the kernel of $\Inv(D)\longrightarrow\Inv(T_1)$ is the set of all ideals $I$ such that $IT_1=T_1$; by \cite[Lemma 8.2]{locpic} (and see \cite[Proposition 5.6]{jaff-derived}), for every such $I$ there are only finitely many $T\in\Theta$ (all of them out of $\njaff(\Theta)$) such that $IT\neq T$. Thus, the kernel $K_1$ is contained in the direct sum $H:=\bigoplus\{\Inv(T)\mid T\in\Theta\setminus\njaff(\Theta)\}$. If now $I_1,\ldots,I_n$ are invertible ideals of $A_1,\ldots,A_n$, respectively (for some $A_1,\ldots,A_n\in\Theta\setminus\njaff(\Theta)$), then $I_1\cap\cdots\cap I_n\cap D$ is an invertible ideal of $D$ with inverse $J_1\cap\cdots\cap J_n\cap D$ (where $J_i$ is the inverse of $I_i$ in $A_i$). Thus $H$ is contained in the kernel, and so $K_1$ has the required decomposition.

Suppose now that the claim holds for all $\beta<\alpha$. If $\alpha=\gamma+1$ is a successor ordinal, then the map $\Inv(D)\longrightarrow\Inv(T_\alpha)$ factors through $\Inv(T_\gamma)$; therefore, there is an exact sequence
\begin{equation}\label{eq:exseq-Kalpha}
0\longrightarrow K_\gamma\longrightarrow K_\alpha\longrightarrow K_{\alpha,\gamma}\longrightarrow 0,
\end{equation}
where $K_{\alpha,\gamma}$ is the kernel of $\Inv(T_\gamma)\longrightarrow\Inv(T_\alpha)$. By the case $\alpha=1$, $K_{\alpha,\gamma}$ is free and isomorphic to $\bigoplus\{\Inv(T)\mid T\in\njaff^{\gamma}(\Theta)\setminus\njaff^\alpha(\Theta)\}$; thus the exact sequence \eqref{eq:exseq-Kalpha} splits, and $K_\alpha\simeq K_\gamma\oplus K_{\alpha,\gamma}$. By induction, $K_\alpha$ is free and has the required decomposition.

Suppose that $\alpha$ is a limit ordinal. Then, the sequence $\{K_\beta\}_{\beta<\alpha}$ is an ascending chain of subgroups of $K_\alpha$, and by the previous reasoning $K_{\beta+1}\simeq K_\beta\oplus H_\beta$, where $H_\beta:=\bigoplus\{\Inv(T)\mid T\in\njaff^{\beta+1}(\Theta)\setminus\njaff^\beta(\Theta)\}$. Let $I\in K_\alpha$. Then, $IT_\alpha=T_\alpha$, and thus by Lemma \ref{lemma:explostion-prejaff} there is a $\beta<\alpha$ such that $IT_\beta\neq T_\beta$ and $IT_{\beta+1}=T_{\beta+1}$. Since $\gamma$ is a limit ordinal, $\beta+1<\gamma$ and $I\in K_{\beta+1}$. Thus $\bigcup_{\beta<\alpha} K_\beta=K_\alpha$, and by \cite[Lemma 5.6]{SP-scattered} (see also \cite[Chapter 3, Lemma 7.3]{fuchs-abeliangroups}) we have that $K_\alpha$ has the required decomposition. By induction, the sequence \eqref{eq:exseq-Talpha} is always exact.

For the last statement, the hypothesis imply that $\njaff^\alpha(\Theta)=\emptyset$ and $T_\alpha=K$, and thus \eqref{eq:exseq-Talpha} becomes
\begin{equation*}
0\longrightarrow\bigoplus_{T\in\Theta}\Inv(T)\longrightarrow\Inv(D)\longrightarrow 0 \longrightarrow 0.
\end{equation*}
The claim follows.
\end{proof}

In the context of one-dimensional domains, we obtain the following.
\begin{cor}
Let $D$ be a one-dimensional domain. If $\Max(D)^\inverse$ is scattered and $\Inv(D_M)$ is free for every $M\in\Max(D)$, then
\begin{equation*}
\Inv(D)\simeq\bigoplus\{\Inv(D_M)\mid M\in\Max(D)\}.
\end{equation*}
In particular, $\Inv(D)$ is free.
\end{cor}
\begin{proof}
The family $\Theta:=\{D_M\mid M\in\Max(D)\}$ is a pre-Jaffard family, and $\njaff^\alpha(\Theta)=\deriv^\alpha(\Theta^\inverse)\simeq\deriv^\alpha(\Max(D)^\inverse)$ \cite[Theorem 8.4]{jaff-derived}. We can now apply Proposition \ref{prop:preJaff}.
\end{proof}

\begin{oss}
When the groups $\Inv(D_M)$ are not free, in general we do not have a decomposition $\Inv(D)\simeq\bigoplus_M\Inv(D_M)$. For example, suppose that $D$ is a one-dimensional Pr\"ufer domain such that $\Max(D)^\inverse$ has a unique limit point, say $Q$; suppose that $D_P$ is a DVR for every $P\neq Q$ and that $\Gamma(D_Q)\simeq\insQ$. Taking $\Theta=\{D_M\mid M\in\Max(D)\}$, we have $\njaff(\Theta)=\{D_Q\}$, and thus from the extension map $\Inv(D)\longrightarrow\Inv(D_Q)$ we get an exact sequence
\begin{equation*}
0\longrightarrow F\longrightarrow\Inv(D)\longrightarrow\insQ\longrightarrow 0,
\end{equation*}
where $F$ is a free group. We claim that $\Inv(D)$ does not contain any subgroup isomorphic to $\insQ$.

Indeed, say that an element $g$ of a group $G$ is \emph{divisible} if for every positive integer $n$ there is an $h\in G$ such that $h^n=g$. Then, every element of $\insQ$ is divisible, and thus every group containing $\insQ$ contains nonzero divisible elements.

Suppose that $I\in\Inv(D)$ is divisible. Since $Q$ is a limit point of $\Max(D)^\inverse$, there is a $P\neq Q$ such that $ID_P\neq D_P$. If $n>|v_P(I)|\neq 0$, then $|v_P(J^n)|=n|v_P(J)|$ is either $0$ or at least $n$, and thus cannot be equal to $v_P(I)$. Thus $J^n\neq I$ for every $J$, and $I$ is not divisible. Hence, $\Inv(D)$ cannot contain a group isomorphic to $\insQ$, as claimed.
\end{oss}

\section{Pr\"ufer domains}\label{sect:prufer}
In this section, we analyze the case of Pr\"ufer domains. We start with a very straightforward lemma.
\begin{lemma}
Let $V$ be a valuation domain. Then, $\Inv(V)\simeq\Gamma(V)$.
\end{lemma}
\begin{proof}
Every invertible ideal of a valuation domain is principal; hence, we have a map
\begin{equation*}
\begin{aligned}
\gamma \colon \Inv(V) & \longrightarrow \Gamma(V),\\
xV & \longmapsto v(x).
\end{aligned}
\end{equation*}
It is easy to see that $\gamma$ is well-defined, injective and surjective. Thus $\Inv(V)\simeq\Gamma(V)$, as claimed.
\end{proof}

Thus, Corollary \ref{cor:locfin} immediately becomes:
\begin{cor}\label{cor:prufer-locfin}
Let $D$ be a locally finite Pr\"ufer domain. If $\Gamma(D_M)$ is free for every $M\in\Max(D)$, then $\Inv(D)$ is free.
\end{cor}

Likewise, the results of the previous sections can be specialized to the case of one-dimensional Pr\"ufer domains.
\begin{prop}
Let $D$ be a one-dimensional Pr\"ufer domain.
\begin{enumerate}[(a)]
\item If $D$ is locally finite, then
\begin{equation*}
\Inv(D)\simeq\bigoplus_{M\in\Max(D)}\Gamma(D_M).
\end{equation*}
In particular, $\Inv(D)$ is free if and only if $\Gamma(D_M)$ is free for every $M\in\Max(D)$.
\item If $\Max(D)^\inverse$ is scattered and $\Gamma(D_M)$ is free for every $M\in\Max(D)$, then $\Inv(D)$ is free.
\end{enumerate}
\end{prop}
\begin{proof}
Let $\Theta:=\{D_M\mid M\in\Max(D)\}$. If $D$ is locally finite, then $\Theta$ is a Jaffard family, and the claim follows from Proposition \ref{prop:jaffard}. In general, $\Theta$ is a pre-Jaffard family, and if $\Max(D)^\inverse$ is scattered then $\Theta$ is sharp \cite[Corollary 8.6]{jaff-derived}; the second claim now follows from Proposition \ref{prop:preJaff}.
\end{proof}

We can say something more about valuation domains.
\begin{prop}\label{prop:val-nobranched}
Let $V$ be a valuation domain without branched prime ideals. Suppose that, for every prime ideal $P\neq(0)$, the group $\Gamma(V_P/Q)$ is free, where $Q$ is the prime ideal directly below $P$. Then, $\Gamma(V)$ is free.
\end{prop}
\begin{proof}
Note first that, if $\Inv(W)\simeq\Gamma(W)$ is free, we can always find a basis of integral invertible ideal by passing, if needed, from $xW$ to $x^{-1}W$.

For every prime ideal $P\neq(0)$, let $\mathcal{B}(P)$ be a subset of $\Inv(V)$ such that the image of $\mathcal{B}(P)$ into $D_P/Q$ is a basis for $\Inv(D_P/Q)$; in particular, each element of $\mathcal{B}(P)$ is generated by an element of $P\setminus Q$. Then, $\mathcal{B}(P)$ is an independent subset of $\Inv(V)$, and thus it generates a free group, say $H(P)$. Let $H$ be the group generated by all the $H(P)$.

We claim that $H=\bigoplus_P H(P)$. Indeed, suppose that there are $y_1,\ldots,y_k\in K$ such that $y_iV\in H(P_i)$ and $y_1\cdots y_kV=V$, with $P_i\neq P_j$ if $i\neq j$. We can suppose without loss of generality that  $P_1\supsetneq\cdots\supsetneq P_k$. Then, $y_1\cdots y_kV_{P_k}=V_{P_k}$. If $i<k$ and $xV\in\mathcal{B}(P_i)$, then $x\in P_i\setminus P_k$, and thus $xV_{P_k}=V_{P_k}$: it follows that $y_iV_{P_k}=V_{P_k}$ for all $i<k$, and thus also $y_kV_{P_k}=V_{P_k}$. Since the quotient map $V\longrightarrow V_P/Q$ induces an isomorphism between $H(P)$ and $\Inv(V_P/Q)$, it follows that $y_kV$ is the identity in $\Inv(V)$, i.e., $y_kV=V$. Repeating the process we obtain that $y_iV=V$ for every $i$. Thus the $H(P)$ are independent subgroups, and $H$ is their direct sum.

We now claim that $H=\Inv(V)$. Let $xV\in\Inv(V)$, and let $P$ be the smallest prime ideal containing $x$ (we can suppose $x\in V$). Then, $xV_P=y_0V_P$ for some $y_0\in H(P)$: indeed, $xV_P/Q$ can be written as a product $w_1^{n_1}\cdots w_t^{n_t}$, where $w_i=\pi(z_i)$ for some $z_i\in\mathcal{B}(P)$ (and $\pi:V\longrightarrow V_P/Q$ is the canonical map) and we can take $y_0:=z_1^{n_1}\cdots z_t^{n_t}$. Let $x_1$ be either $xy_0^{-1}$ or $x^{-1}y_0$, according to which is in $V$: if $x_1V\neq V$, then $x_1V_P=V_P$, and thus the minimal prime $P_1$ over $x_1V$ properly contains $P$. Continuing in this way, we obtain an ascending chain $P\subsetneq P_1\subsetneq P_2\subsetneq \cdots$ of prime ideals; however, since there are no branched prime ideals, such a chain must stop. (The union of an ascending chain of prime ideals is an unbranched prime.) Thus, at one point we must have $x_nV=V$, and thus $x_nV\in H$; by construction, also $xV\in H$. Hence $H=\Inv(V)$, as claimed.

In particular, as a direct sum of free groups, $\Inv(V)\simeq\Gamma(V)$ is free.
\end{proof}

\begin{cor}\label{cor:val-stronglydisc}
Let $D$ be a locally finite strongly discrete Pr\"ufer domain. Then, $\Inv(D)$ is free.
\end{cor}
\begin{proof}
Suppose first that $D=V$ is a valuation domain: then, no prime ideal of $V$ is unbranched. Moreover, for every prime $P$, since $P\neq P^2$ then $\Gamma(V_P/Q)\simeq\insZ$ (where $Q$ is the prime ideal directly below $P$) and thus it is free. The claim follows from Proposition \ref{prop:val-nobranched}.

If $D$ is locally finite, the claim now follows from the previous part of the proof and Proposition \ref{prop:oplus-tcomplete}.
\end{proof}

We now want to use the properties of Pr\"ufer domains to study more deeply the conditions under which $\Inv(D)$ is free, and to extend as far as possible those results to the group $\Div(D)$ of $v$-invertible ideals. The idea is to use quotients by divided primes. Recall that a prime ideal $P$ is \emph{divided} if $P=PD_P$.

Let $D$ be a Pr\"ufer domain, $\star$ a star operation on $D$, and let $P$ be a non-maximal divided prime ideal. Let $\pi:D\longrightarrow D/P$ be the quotient. Then, $\star$ induces a star operation $\sharp$ on $D/P$, defined by
\begin{equation*}
I^\sharp:=\pi(\pi^{-1}(I)^\star)
\end{equation*}
for every fractional ideal $I$ of $D$ (see \cite{fontana-park} and \cite[Section 6]{starloc}). The sets of $\star$-invertible and $\sharp$-invertible ideals are closed related.
\begin{prop}\label{prop:divided-invstar}
Let $D$ be a Pr\"ufer domain and let $P$ be a non-maximal divided prime ideal. Let $\star$ be a star operation on $D$ and let $\sharp$ be the corresponding star operation on $R:=D/P$. Then, there is an exact sequence
\begin{equation*}
0\longrightarrow \Inv^\sharp(R)\longrightarrow \Inv^\star(D)\longrightarrow \Gamma(D_P)\longrightarrow 0.
\end{equation*}
In particular, if $\Gamma(D_P)$ is free, then $\Inv^\star(D)$ is free if and only if $\Inv^\sharp(R)$ is free.
\end{prop}
\begin{proof}
Let $I$ be a $\star$-invertible ideal: we claim that $v_P(I)$ has a minimum in $\Gamma(D_P)$. Indeed, if $I$ is $\star$-invertible then it is also $v$-invertible, and thus $(I:I)=D$. Since $PD_P=P\subseteq D$, if $x\in I$ and $p\in P$ we have $xp\in I$; therefore, if $\gamma\in v_P(I)$ then $I$ contains all $y$ such that $v_P(y)>\gamma$. Suppose $v_P(I)$ has not a minimum, and let $x\in I$. Then, there is an $x'\in I$ such that $v_P(x')<v_P(x)$, and thus $I$ contains all elements $y$ such that $v_P(y)=v_P(x)$; in particular, $xD_P\subseteq I$, and since $x$ was arbitrary we would have $ID_P\subseteq I$, against $(I:I)=D$. Thus $v_P(I)$ has a minimum.

In particular, the map
\begin{equation*}
\begin{aligned}
\pi\colon \Inv^\star(D) &\longrightarrow \Gamma(D_P),\\
I & \longmapsto \min v_P(I)
\end{aligned}
\end{equation*}
is well-defined and a surjective group isomorphism. Its kernel is $\ker\pi=\{I\in\Inv^\star(D)\mid P\subsetneq I\subseteq D_P\}$; by the proof of \cite[Proposition 7.3]{starloc}, the quotient $D\longrightarrow R$ induces an isomorphism between this set and $\Inv^\sharp(R)$. The existence of the exact sequence follows.

For the ``in particular'' statement, if $\Gamma(D_P)$ is free then the sequence splits and $\Inv^\star(D)\simeq\Inv^\sharp(R)\oplus\Gamma(D_P)$. In particular, $\Inv^\star(D)$ is free if and only if $\Inv^\sharp(R)$ is free.
\end{proof}

\begin{cor}\label{cor:divided-invstar:free}
Let $D$ be a Pr\"ufer domain and let $P$ be a non-maximal divided prime ideal. Suppose that $\Gamma(D_P)$ is free. Then:
\begin{enumerate}[(a)]
\item $\Inv(D)$ is free if and only if $\Inv(D/P)$ is free;
\item $\Div(D)$ is free if and only if $\Div(D/P)$ is free.
\end{enumerate}
\end{cor}
\begin{proof}
Both statements follow from Proposition \ref{prop:divided-invstar}: the first one using $\star=d$ (so $\sharp=d$), the second one by using $\star=v$ (and thus also $\sharp=v$).
\end{proof}

Let $D$ be a Pr\"ufer domain. Following \cite{starloc}, we say that a prime $P$ of $D$ is a \emph{branching point} for $\Spec(D)$ if there is a family $\Delta\subseteq\Spec(D)$ of pairwise incomparable primes, each one strictly larger than $P$, such that $P=\inf\Delta$ (in the containment order). We denote by $\Spechi(D)$ the union of $(0)$, $\Max(D)$ and the branching points of $\Spec(D)$, and we call it the \emph{homeomorphically irreducible tree} associated to $\Spec(D)$.
\begin{lemma}\label{lemma:branchpoint-max}
Let $D$ be a Pr\"ufer domain and let $P$ be a nonmaximal prime ideal. Then, $P$ is a branching point for $\Spec(D)$ if and only if $P=\inf(V(P)\cap\Max(D))$.
\end{lemma}
\begin{proof}
If $P=\inf(V(P)\cap\Max(D))$, then we can use $V(P)\cap\Max(D)$ as the set in the definition of a branching point. Conversely, suppose that $P\neq\inf(V(P)\cap\Max(D))=:Q$ (note that $Q$ exists since $\Spec(D)$ is a tree). Then, $P\subseteq Q$, and thus $V(P)=\Delta\cup V(Q)$, where $\Delta$ is the set of prime ideals between $P$ and $Q$; note that $\Delta$ is linearly ordered, and each element of $\Delta$ is contained in all elements of $V(Q)$. Let $\Lambda\subseteq V(P)\setminus\{P\}$ be a set of pairwise incomparable elements. If $\Delta\cap\Lambda=\emptyset$, then $\Lambda\subseteq V(Q)$ and $\inf\Lambda\supseteq Q$; if $\Delta\cap\Lambda\neq\emptyset$, then by construction $\Lambda$ must be a singleton $\{L\}$, and so $\inf\Lambda=L$. Therefore, $\inf\Lambda\neq P$ for every such $\Lambda$, and thus $P$ is not a branching point.
\end{proof}

\begin{lemma}\label{lemma:Spechi-semiloc}
Let $D$ be a semilocal Pr\"ufer domain.  Then, $\Spechi(D)$ is finite.
\end{lemma}
\begin{proof}
By Lemma \ref{lemma:branchpoint-max}, a branching point is uniquely defined by a subset of $\Max(D)$. If $D$ is semilocal, then $\Max(D)$ is finite, and thus $\Spec(D)$ has only finitely many branching points. It follows that $\Spechi(D)$ is finite.
\end{proof}

Let $D$ be a Pr\"ufer domain. We say that two maximal ideals $P,Q$ are \emph{dependent} if $P\cap Q$ contains a nonzero prime ideal, or equivalently if $D_PD_Q\neq K$; the fact that $\Spec(D)$ is a tree implies that dependence is an equivalence relation. For an equivalence class $\Delta$, let $T(\Delta):=\bigcap\{D_P\mid P\in\Delta\}$; we call the family $\{T(\Delta)\}$ (as $\Delta$ ranges among the equivalence classes) the \emph{standard decomposition} of $D$. If $D$ is semilocal, or if $\Spec(D)$ is a Noetherian space, the standard decomposition is a Jaffard family \cite[Proposition 6.2]{starloc}; if $D$ has finite dimension, then the members of the standard decomposition are in bijective correspondence with the height-1 primes of $D$ (see \cite[Lemma 6.1]{starloc}).

\begin{prop}\label{prop:cutbranch-semiloc-inv}
Let $D$ be a semilocal Pr\"ufer domain. Suppose that $\Gamma(D_P)$ is free for every $P\in\Spechi(D)\setminus\Max(D)$. Then, $\Inv(D)$ is free if and only if $\Gamma(D_M)$ is free for every $M\in\Max(D)$.
\end{prop}
\begin{proof}
If each $\Inv(D_M)$ is free, the claim follows from Corollary \ref{cor:prufer-locfin}.

Suppose that $\Inv(D)$ is free: we have to prove that each $\Inv(D_M)$ is free. Since $D$ is semilocal, $\Spechi(D)$ is finite by Lemma \ref{lemma:Spechi-semiloc}. We proceed by induction on its cardinality.

If $|\Spechi(D)|=1$ then $D$ is a field, and $\Inv(D)=\Gamma(D)$ is trivial. If $|\Spechi(D)|=2$ then $D$ is a valuation domain and the claim follows from the fact that $\Inv(D)\simeq\Gamma(D)$.

Suppose now that the claim holds up to $n-1$. Consider the standard decomposition $\Theta$ of $D$; since $D$ is semilocal, $\Theta$ is a Jaffard family of $D$ and thus $\Inv(D)\simeq\bigoplus\{\Inv(T)\mid T\in\Theta\}$.

If $|\Theta|>1$, then $|\Spechi(T)|<|\Spechi(D)|$ for every $T\in\Theta$, and thus the inductive hypothesis applies to each $T$. Therefore, each $\Inv(T)$ and thus by induction $\Gamma(T_N)$ is free for every $T\in\Theta$ and every $N\in\Max(T)$. However, the set of these $T_N$ is just the set of all $D_M$ as $M$ ranges in $\Max(D)$; thus each $\Inv(D_M)$ is free.

If $|\Theta|=1$, then $P:=\inf\Max(D)\in\Spechi(D)$ is nonzero; since $P$ is contained in every maximal ideal of $D$, moreover, $P$ is divided. Since $\Gamma(D_P)$ is free by hypothesis, by Corollary \ref{cor:divided-invstar:free} $\Inv(D/P)$ is free. However, $|\Spechi(D/P)|=|\Spechi(D)|-1$, since the elements of $\Spechi(D/P)$ are exactly the quotients of the nonzero elements of $\Spechi(D)$; hence, $\Gamma((D/P)_N)$ is free for every $N\in\Max(D/P)$. For every maximal ideal $M$ of $D$, we have $(D/P)_{M/P}\simeq D_M/PD_M$; hence, we have an exact sequence
\begin{equation*}
0\longrightarrow\Gamma((D/P)_{M/P})\longrightarrow\Gamma(D_M)\longrightarrow\Gamma(D_P)\longrightarrow 0,
\end{equation*}
that splits since $\Gamma(D_P)$ is free. Thus $\Gamma(D_M)$ is free, and the statement follows. By induction, the claim holds for every semilocal Pr\"ufer domain.
\end{proof}

\begin{cor}\label{cor:cutbranch-locfin-inv}
Let $D$ be Pr\"ufer domain that is finite-dimensional and locally finite. Suppose that $\Gamma(D_P)$ is free for every $P\in\Spechi(D)\setminus\Max(D)$. Then, $\Inv(D)$ is free if and only if $\Gamma(D_M)$ is free for every $M\in\Max(D)$.
\end{cor}
\begin{proof}
For every $P\in X^1(D)$, let $T(P):=\bigcap\{D_Q\mid Q\in V(P)\}$. Then, $\Theta:=\{T(P)\mid P\in X^1(D)\}$ is the standard decomposition of $D$, and it is a Jaffard family of $D$ since $D$ is locally finite and finite-dimensional. Thus $\Inv(D)\simeq\bigoplus\{\Inv(T)\mid T\in\Theta\}$. Moreover, each $T\in\Theta$ is semilocal, and thus by Proposition \ref{prop:cutbranch-semiloc-inv} $\Inv(T)$ is free if and only if $\Gamma(T_N)$ is free for every $N\in\Max(T)$. Putting all together, it follows that $\Inv(D)$ is free if and only if each $\Gamma(D_M)$ is free, as claimed.
\end{proof}

We now turn to study the group $\Div(D)$: the first result deals with valuation domains.
\begin{prop}\label{prop:div-val-b}
Let $V$ be a valuation domain with maximal ideal $M$. Suppose that $M$ is branched, and let $P$ be the prime ideal directly below $M$. Then, the following hold.
\begin{enumerate}[(a)]
\item\label{prop:div-val-b:fg} If $M$ is finitely generated, then $\Div(V)\simeq\Gamma(V)$.
\item\label{prop:div-val-b:nfg} If $M$ is not finitely generated, then $\Div(V)\simeq\insR\oplus\Gamma(V_P)$. In particular, $\Div(V)$ is not free.
\end{enumerate}
\end{prop}
\begin{proof}
\ref{prop:div-val-b:fg} By \cite[Corollary 3.6]{afz_vclass}, if $M$ is principal then every $v$-invertible $v$-ideal is principal. Thus $\Div(V)=\Inv(V)\simeq\Gamma(V)$, as claimed.

\ref{prop:div-val-b:nfg} By Proposition \ref{prop:divided-invstar}, we have an exact sequence
\begin{equation*}
0\longrightarrow\Div(V/P)\longrightarrow\Div(V)\longrightarrow\Gamma(V_P)\longrightarrow 0.
\end{equation*}
Since $M$ is not finitely generated, then by the proof of \cite[Theorem 2.7]{afz_vclass} $\Div(V/P)\simeq\insR$. In particular, it is a divisible group, and thus the above exact sequence splits; it follows that $\Div(V)\simeq\insR\oplus\Gamma(V_P)$. Since $\insR$ is not free, neither is $\Div(V)$.
\end{proof}

The proof of the following proposition follows a path similar to the proof of Proposition \ref{prop:cutbranch-semiloc-inv}.
\begin{prop}\label{prop:cutbranch-semiloc-div}
Let $D$ be a semilocal Pr\"ufer domain. Suppose that every maximal ideal is branched and that $\Gamma(D_P)$ is free for every $P\in\Spechi(D)\setminus\Max(D)$. Then, $\Div(D)$ is free if and only if every maximal ideal of $D$ is finitely generated.
\end{prop}
\begin{proof}
Since $D$ is semilocal, $\Spechi(D)$ is finite by Lemma \ref{lemma:Spechi-semiloc}. We proceed by induction on its cardinality.

If $|\Spechi(D)|=1$ then $D$ is a field, $\Div(D)$ is trivial and its maximal ideal is the zero ideal, which is finitely generated. If $|\Spechi(D)|=2$ then $D$ is a valuation domain; the claim now follows from Proposition \ref{prop:div-val-b}.

Suppose now that the claim holds up to $n-1$. Consider the standard decomposition $\Theta$ of $D$; since $D$ is semilocal, $\Theta$ is a Jaffard family of $D$ and thus $\Div(D)\simeq\bigoplus\{\Div(T)\mid T\in\Theta\}$.

If $|\Theta|>1$, then $|\Spechi(T)|<|\Spechi(D)|$, and thus the inductive hypothesis applies to each $T$. If every maximal ideal of $D$ is finitely generated, so is every maximal ideal of $T$: thus each $\Div(T)$ is free and $\Div(D)$ is free. Conversely, if $\Div(D)$ is free then each $\Div(T)$ is free, and so every maximal ideal of $T$ is finitely generated. If $M\in\Max(D)$, then $M=MT\cap D$ for some $T\in\Theta$; by \cite[Lemma 5.9]{starloc}, since $T$ is a Jaffard overring of $D$ and $MT$ is finitely generated also $M$ is finitely generated, as claimed.

If $|\Theta|=1$, then $P:=\inf\Max(D)\in\Spechi(D)$ is nonzero; since $P$ is contained in every maximal ideal of $D$, moreover, $P$ is divided. Since $\Gamma(D_P)$ is free by hypothesis, by Corollary \ref{cor:divided-invstar:free} $\Div(D)$ is free if and only if $\Div(D/P)$ is free. However, $|\Spechi(D/P)|=|\Spechi(D)|-1$, since the elements of $\Spechi(D/P)$ are exactly the quotients of the nonzero elements of $\Spechi(D)$; hence, $\Div(D/P)$ is free if and only if all the maximal ideals of $D/P$ are finitely generated. Since every maximal ideal $M$ of $D$ contains $P$, we have that $M$ is finitely generated if and only $M/P$ is finitely generated; thus $\Div(D/P)$ is free if and only if every maximal ideal of $D$ is finitely generated. The claim is proved.

By induction, the equivalence holds for every semilocal Pr\"ufer domain.
\end{proof}

\begin{cor}
Let $D$ be Pr\"ufer domain that is finite-dimensional and locally finite. Suppose that every maximal ideal is branched and that $\Gamma(D_P)$ is free for every $P\in\Spechi(D)\setminus\Max(D)$. Then, $\Div(D)$ is free if and only if every maximal ideal is finitely generated.
\end{cor}
\begin{proof}
Let $\Theta$ be the family defined in the proof of Corollary \ref{cor:cutbranch-locfin-inv}; then, $\Div(D)$ is free if and only if $\Div(T)$ is free for every $T\in\Theta$. Since each $T\in\Theta$ is a Jaffard overring, $N\in\Max(T)$ is finitely generated if and only if $N\cap D\in\Max(D)$ is finitely generated; using Proposition \ref{prop:cutbranch-semiloc-div} we see that $\Div(D)$ is free if and only if each maximal ideal is finitely generated, as claimed.
\end{proof}

Proposition \ref{prop:divided-invstar} can also be used to extend results proved by other means. For example, suppose that $D$ is a strongly discrete $2$-dimensional Pr\"ufer domain with a single height-one prime ideal $P$. Then, $\Gamma(D_P)\simeq\insZ$ is free, while $D/P$ is an almost Dedekind domain and thus $\Inv(D/P)$ is free by \cite[Proposition 5.3]{bounded-almded}. Using quotients (i.e., Proposition \ref{prop:divided-invstar}) we have that also $\Inv(D)$ is free. Extending this reasoning, we have the following result, where $X^k(D)$ indicates the set of prime ideals of $D$ of height $k$.
\begin{prop}
Let $D$ be a strongly discrete Pr\"ufer domain of dimension $d<\infty$. If $\Spec(D)\setminus X^d(D)$ is finite, then $\Inv(D)$ is free.
\end{prop}
\begin{proof}
If $d=1$, then $D$ is an almost Dedekind domain, and the claim follows from \cite[Proposition 5.3]{bounded-almded}. Suppose that $d>1$ and that the claim is true up to dimension $d-1$, and let $\Theta$ be the standard decomposition of $D$. Since $D$ is finite-dimensional, $\Theta$ is in bijective correspondence with the elements of $X^1(D)\subseteq\Spec(D)\setminus X^d(D)$, and thus it is finite; hence $\Theta$ is a Jaffard family of $D$ and $\Inv(D)\simeq\bigoplus\{\Inv(T)\mid T\in\Theta\}$.

Any $T\in\Theta$ has a unique height-1 prime ideal, say $P$; then, $D_P$ is a DVR since $D$ is strongly discrete. If $P$ is maximal, then $T=D_P$ and $\Inv(D_P)\simeq\Gamma(D_P)\simeq\insZ$ is free. If $P$ is not maximal, then $P$ is divided; moreover, $D/P$ is a strongly discrete Pr\"ufer domain of dimension $d-1$, and $\Spec(D/P)\setminus X^{d-1}(D/P)$ is finite since it is in bijective correspondence with $V(P)\setminus X^d(D)\subseteq\Spec(D)\setminus X^d(D)$. By hypothesis, $\Inv(D/P)$ is free; by Corollary \ref{cor:divided-invstar:free}, it follows that $\Inv(D)$ is free. By induction, the claim is proved.
\end{proof}

\begin{oss}
The condition in the statement is stronger than requiring that $\Spec(D)\setminus\Max(D)$ is finite, since there may be maximal ideals of height lower than $d$. If we only required $|\Spec(D)\setminus\Max(D)|<\infty$, we would not be able to prove that the standard decomposition is finite, as there may be infinitely many maximal ideals of height $1$.
\end{oss}

The previous proposition and the earlier Corollary \ref{cor:val-stronglydisc} suggest the following
\begin{conj}\label{conj:stronglydiscrete}
If $D$ is a strongly discrete Pr\"ufer domain, then $\Inv(D)$ is free.
\end{conj}
We shall see another case of this conjecture in the next section.

\section{Algebras}\label{sect:algebras}
Let $D\subseteq R$ be an extension of integral domains. Then, the inclusion map $D\longrightarrow R$ induces maps $\insprinc(D)\longrightarrow\insprinc(R)$ and $\Inv(D)\longrightarrow\Inv(R)$. It also induces a natural group homomorphism $\phi_p:\Pic(D)\longrightarrow\Pic(R)$ of their Picard groups. Following \cite{locpic}. we call the quotient $\Pic(R)/\phi_p(\Pic(D))$ the \emph{local Picard group} of the extension $D\subseteq R$, and we denote it by $\locpic(R,D)$.

In this section, we use these facts to study the group of invertible ideals of $R$. To simplify the notation, we use $F^\nz$ to denote the unit group of a field $F$.

\begin{teor}\label{teor:ext-retract}
Let $D$ be an integral domain with quotient field $K$, and let $R$ be a $D$-algebra that extends $D$ and is an integral domain; let $L$ be the quotient field of $R$. Suppose that:
\begin{itemize}
\item $\unit(D)=\unit(R)$;
\item the quotient $L^\nz/K^\nz$ is free;
\item $\locpic(R,D)$ and $\Inv(D)$ are free.
\end{itemize}
Then, $\Inv(R)$ is free.
\end{teor}
\begin{proof}
Consider the commutative diagram
\begin{equation*}
\begin{tikzcd}
0\arrow[r] & \unit(D)\arrow[r]\arrow[d] & K^\nz\arrow[r]\arrow[d] & \insprinc(D)\arrow[r]\arrow[d] & 0\\
0\arrow[r] & \unit(R)\arrow[r] & L^\nz\arrow[r] & \insprinc(R)\arrow[r] & 0.
\end{tikzcd}
\end{equation*}
By hypothesis, the leftmost vertical map is an equality, while the middle vertical map is injective; from the snake lemma, it follows that the rightmost map $\insprinc(D)\longrightarrow\insprinc(R)$ is injective and that its cokernel is equal to the cokernel of $K^\nz\longrightarrow L^\nz$. Thus, we have an exact sequence
\begin{equation*}
0\longrightarrow \insprinc(D)\longrightarrow\insprinc(R)\longrightarrow L^\nz/K^\nz\longrightarrow 0,
\end{equation*}
that extends to a commutative diagram
\begin{equation*}
\begin{tikzcd}
& 0\arrow[d] & 0\arrow[d] & 0\arrow[d]\\
0\arrow[r] & \insprinc(D)\arrow[r]\arrow[d] & \insprinc(R)\arrow[r]\arrow[d] & L^\nz/K^\nz\arrow[r]\arrow[d] & 0\\
0\arrow[r] & \Inv(D)\arrow[r]\arrow[d] & \Inv(R)\arrow[r]\arrow[d] & G\arrow[r]\arrow[d] & 0\\
0\arrow[r] & \Pic(D)\arrow[r]\arrow[d] & \Pic(R)\arrow[r]\arrow[d] & \locpic(R,D)\arrow[r]\arrow[d] & 0\\
& 0 & 0 & 0
\end{tikzcd}
\end{equation*}
where $G$ is defined as the cokernel of the extension map $\Inv(D)\longrightarrow\Inv(R)$.

By hypothesis, $\locpic(R,D)$ is free; therefore, the rightmost column is a split exact sequence, and thus $G\simeq\locpic(R,D)\oplus L^\nz/K^\nz$ is free (using the hypothesis on $L^\nz/K^\nz$). Hence, also the middle row also splits. By hypothesis, $\Inv(D)$ is free; therefore, $\Inv(R)\simeq\Inv(D)\oplus G$ is free.
\end{proof}

We now show that, for rings of polynomials, the first two hypothesis of Theorem \ref{teor:ext-retract} often holds.
\begin{lemma}\label{lemma:quoz-unit-fields}
Let $K$ be a field, and let $L$ be a field containing $K$. If there is a Krull domain $R$ containing $K$ with quotient field $L$ such that $\unit(R)=K^\nz$, then $L^\nz/K^\nz$ is free. In particular, if $\XX$ is a set of indeterminates, then $K(\XX)^\nz/K^\nz$ is free.
\end{lemma}
\begin{proof}
Let $\phi:L^\nz\longrightarrow\insprinc(R)$ be the canonical map: then, $\ker\phi=\unit(R)=K^\nz$. and thus $L^\nz/K^\nz\simeq\insprinc(D)$. Since $R$ is a Krull domain, $\insprinc(R)$ is free, and thus $L^\nz/K^\nz$ is free.

The last statement follows by considering $R=K[\XX]$.
\end{proof}

\begin{cor}
Let $D$ be an integral domain such that $\Inv(D)$ is free. Then:
\begin{enumerate}[(a)]
\item if $\locpic(D[X],D)$ is free, then $\Inv(D[X])$ is free;
\item if $D$ is seminormal, then $\Inv(D[X])$ is free.
\end{enumerate}
\end{cor}
\begin{proof}
The first point is a direct consequence of Theorem \ref{teor:ext-retract} and Lemma \ref{lemma:quoz-unit-fields}. The second one follows from the first one since $\locpic(D[X],D)$ is trivial if $D$ is seminormal \cite[Theorem 1.6]{pic-RX-seminormal}.
\end{proof}

The \emph{ring of integer-valued polynomials} over a domain $D$, denoted by $\Int(D)$, is the ring of all polynomials $f\in K[X]$ (where $K$ is the quotient field of $D$) such that $f(D)\subseteq D$. We say that $\Int(D)$ \emph{behaves well under localization} if $\Int(D_P)=(D\setminus P)^{-1}\Int(D)$ for every prime ideal $P$. See \cite{intD} for information about integer-valued polynomials.
\begin{cor}\label{cor:almded-int}
Let $D$ be an almost Dedekind domain such that $\Int(D)$ behaves well under localization and $\Max(D)$ is scattered when endowed with the inverse topology. Then, $\Inv(\Int(D))$ is free.
\end{cor}
\begin{proof}
The ring $\Int(D)$ is a $D$-algebra contained in $K[X]$; thus $\unit(D)=\unit(R)$ and $K(X)^\nz/K^\nz$ is free. The group $\Inv(D)$ is free since $D$ is almost Dedekind \cite[Proposition 5.3]{bounded-almded}. Under the hypothesis in the statement, $\locpic(\Int(D),D)$ is free \cite[Corollary 7.6]{locpic}. The claim follows from Theorem \ref{teor:ext-retract}.
\end{proof}

\begin{cor}\label{cor:dedekind-int}
Let $D$ be a Dedekind domain. Then, $\Inv(\Int(D))$ is free.
\end{cor}
\begin{proof}
The hypothesis of the previous corollary hold, in particular, if $D$ is a Dedekind domain.
\end{proof}

We note that, when $D$ is an almost Dedekind domain, the ring $\Int(D)$ is strongly discrete; thus the two previous corollaries give some more weight to the conjecture advanced at the end of Section \ref{sect:prufer}.

Corollaries \ref{cor:almded-int} and \ref{cor:dedekind-int} use in their proof the fact that the local Picard group is free; this is proved by localizing $\locpic(R,D)$ to the maximal ideals of $R$. We now show that, in some cases, we can use similar results for $\Inv(R)$ with a more direct approach.
\begin{defin}
Let $D$ be a domain with quotient field $K$, and let $R$ be a $D$-algebra containing $D$. We say that an integral ideal $I$ of $R$ is \emph{unitary} if $I\cap D\neq(0)$. We denote by $\Inv_u(R,D)$ the subgroup of $\Inv(D)$ generated by the unitary ideals of $D$.
\end{defin}

\begin{teor}\label{teor:jaffard-unitary}
Let $D$ be an integral domain and let $\Theta$ be a Jaffard family of $D$. Let $R$ be an integral $D$-algebra. Then,
\begin{equation*}
\Inv_u(R,D)\simeq\bigoplus_{T\in\Theta}\Inv_u(RT,T).
\end{equation*}
\end{teor}
\begin{proof}
We first note that, if $I\in\Inv_u(R,D)$, then $IT\in\Inv_u(RT,T)$, since $IT\cap T$ contains any element of $I\cap D$. Thus, we have extension maps $\phi_T:\Inv_u(R,D)\longrightarrow\Inv_u(RT,T)$, $\phi_T(I)=IT$, combining which we have a map
\begin{equation*}
\begin{aligned}
\Phi\colon\Inv_u(R,D)  & \longrightarrow \prod_{T\in\Theta}\Inv_u(RT,T),\\
I & \longmapsto (IT)_{T\in\Theta}.
\end{aligned}
\end{equation*}
We claim that $\Phi$ is injective and that its range is exactly the direct sum.

To show that it is injective, consider the map $\star:I\mapsto\bigcap\{IRT\mid T\in\Theta\}$. Since $\bigcap\{RT\mid T\in\Theta\}=R$, the map $\star$ is a star operation on $R$, and thus $J^\star=J$ for every invertible ideal $J$ of $R$. In particular, if $J\in\ker\Phi$ then
\begin{equation*}
J=J^\star=\bigcap_{T\in\Theta}JRT=\bigcap_{T\in\Theta}RT=R.
\end{equation*}
Thus $\Phi$ is injective.

Let $I\in\Inv_u(R,D)$. Then, $I=JL^{-1}$ for some unitary invertible ideals $J,L$. There are at most finitely many $T\in\Theta$ such that $JT\neq T$, and finitely many $S\in\Theta$ such that $LS\neq S$; hence, $JT=T=LT$ for all but finitely many $T\in\Theta$, and so $IT=T$ for almost all $T$. Hence the range of $\Phi$ is contained in the direct sum.

To show that the range is $\Phi$ is the whole direct sum, it is enough to show that for every $J\in\Inv_u(RT,T)$ there is an $I\in\Inv_u(R,D)$ such that $IRT=J$ and $IRS=RS$ for all $S\in\Theta$, $S\neq T$. Writing $J=J_1(J_2)^{-1}$, we can suppose without loss of generality that $J\subseteq T$. We claim that $I=J\cap R$ is the right choice. Indeed, by \cite[Lemma 7.2]{locpic}, we have $IT=(J\cap R)T=JT\cap RT=J$ and $IS=(J\cap R)S=JS\cap RS=RS$ if $S\neq T$, since $JS=JRTS=JRK=RK$ as $J$ is unitary. We show that $I$ is invertible: by \cite[Lemma 7.3]{locpic}, $I$ is finitely generated. Let $M$ be a maximal ideal of $R$. If $M\cap D=(0)$, then $R_M$ contains $RK$ and thus $RS$ for all $S\in\Theta$; in particular, if $S\neq T$, then $IR_M\supseteq IRS=RS$, and so $IR_M=R_M$ is principal. If $M\cap D=P\neq(0)$, then $R_M$ contains $D_P$: therefore, $R_M$ contains $RS$, where $S\in\Theta$ is such that $PS\neq S$. If $S\neq T$ the claim holds as above. If $S=T$, then $R_M$ is a localization of $RT$, and $IR_M=JR_M$; since $J$ is invertible, it is locally principal, and thus $IR_M$ is principal. Therefore, $I$ is finitely generated and locally principal, and thus it is invertible. Hence the range of $\Phi$ is the direct sum, and the claim follows.
\end{proof}

The previous theorem requires to look at the distance between $\Inv(R)$ and $\Inv_u(R,D)$.
\begin{prop}\label{prop:quoz-unitary}
Let $D$ be an integral domain with quotient field $K$ and let $R$ be an integral $D$-algebra. Then, $\displaystyle{\frac{\Inv(R)}{\Inv_u(R,D)}\simeq\Inv(RK)}$.
\end{prop}
\begin{proof}
Consider the extension map $\phi:\Inv(R)\longrightarrow\Inv(RK)$, $\phi(I)=IRK$; we claim that $\ker\phi=\Inv_u(R,D)$.

Indeed, if $I$ is unitary then $IRK=K$ since $I$ contains some $d\in D$, $d\neq 0$; thus every unitary ideal is in $\ker\phi$, and so all of $\Inv_u(D)$ is contained in $\ker\phi$. Conversely, if $I\in\ker\phi$, then $IRK=RK$, and thus $I\cap S\neq(0)$, where $S:=D\setminus\{0\}$. If $J$ is the inverse of $I$, then also $J\in\ker\phi$, and thus $JRK=RK$ and $J\cap S\neq(0)$; however, $J=(R:I)$, and thus there is a $d\in D$, $d\neq(0)$ such that $dI\subseteq R$. In particular, $dI$ is unitary (if $d'\in I\cap D$, $d'\neq 0$, then $dd'\in dI\cap D$), and thus $I=(dR)^{-1}(dI)\in\Inv_u(R,D)$. Hence $\ker\phi=\Inv_u(R,D)$, and the claim follows.
\end{proof}

\begin{prop}
Let $D$ be an integral domain and let $\Theta$ be a Jaffard family of $D$. Let $R$ be an integral $D$-algebra. If $\Inv(RT)$ is free for every $T\in\Theta$ and $\Inv(RK)$ is free, then also $\Inv(R)$ is free.
\end{prop}
\begin{proof}
By Proposition \ref{prop:quoz-unitary}, $\frac{\Inv(R)}{\Inv_u(R,D)}\simeq\Inv(RK)$; in particular, if $\Inv(RK)$ is free then $\Inv(R)\simeq\Inv_u(R,D)\oplus\Inv(R,K)$, and thus we only need to show that $\Inv_u(R,D)$ is free. By Theorem \ref{teor:jaffard-unitary}, $\Inv_u(R,D)\simeq\bigoplus\{\Inv_u(RT,T)\mid T\in\Theta\}$, and each $\Inv_u(RT,T)$ is free since it is a subgroup of the free group $\Inv(RT)$. Hence $\Inv_u(R)$ is free and so is $\Inv(R)$.
\end{proof}

Note that these results allow to partially circumvent local finiteness: indeed, $\Theta$ is a locally finite family, but in general $R\cdot\Theta:=\{RT\mid T\in\Theta\}$ is not locally finite. For example, if $D=\insZ$, $\Theta:=\{D_M\mid M\in\Max(D)\}$ and $R=\Int(\insZ)$, then $R\cdot\Theta$ is not locally finite (indeed, all nonconstant polynomials survive in each $\Int(\insZ)\insZ_{(p)}=\Int(\insZ_{(p)})$).

We conclude the paper with an example that is very far from the other cases we considered.
\begin{ex}
Let $E$ be the ring of all entire functions, i.e., the ring of all functions $\insC\longrightarrow\insC$ that are analytical everywhere. The ring $E$ is an infinite-dimensional B\'ezout domain such that every nonzero prime is contained in only one maximal ideal; we refer to \cite{henriksen_prime} and \cite[Section 8.1]{fontana_libro} for generalities on this ring. We claim that $\Inv(E)$ is not free.

Indeed, let $f\in E$ be a function with infinitely many zeros. Then, $f$ induces a Jaffard family $\{E_1,E_2\}$, where $E_1:=\bigcap\{E_M\in\Max(E)\mid f\in M\}$ and $E_2:=\bigcap\{E_M\in\Max(E)\mid f\notin M\}=E[1/f]$, and so $\Inv(E)\simeq\Inv(E_1)\oplus\Inv(E_2)$. We claim that $\Inv(E_1)$ is not free.

Note first that, since $E$ is a B\'ezout domain, so is $E_1$, and thus $\Inv(E_1)=\insprinc(E_1)$. For every $g\in Q(E)$ and every $x\in\insC$, let $v_x(g)$ be the order of the zero of $g$ in $x$ (with $v_x(g)$ negative if $x$ is a pole), or equivalently let $v_x(g)$ be the order of $g$ in the ring $E_{(X-x)}$ (which is a discrete valuation ring). Let $Z(g)$ be the set of zeros of a $g\in E$, and set $Z:=Z(f)$. Consider the map
\begin{equation*}
\begin{aligned}
\psi\colon \Inv(E_1) & \longrightarrow \prod_{x\in Z}\insZ,\\
(g) & \longmapsto (v_x(g))_{x\in Z}.
\end{aligned}
\end{equation*}
It's clear that $\psi$ is a group homomorphism; we claim that it is actually an isomorphism.

To show that it is surjective, it is enough to note that, since $Z$ is discrete (being the zero set of $f$), by Weierstrass' theorem (see e.g. \cite[Chapter 5, Theorem 7]{ahlfors}) for every sequence $(e_x)_{x\in Z}$ of integers there is an entire function $h$ such that $v_x(h)=e_x$; then, $\psi(h)=(e_x)_{x\in Z}$.

Suppose now that $h\in\ker\psi$. Then, $h$ has no zeros nor poles in $Z$; write $h=h_1/h_2$, where $h_1,h_2\in E$ have no common zero. Let $M$ be a maximal ideal containing $f$; then, the set $Z(M):=\{Z(g)\mid g\in E\}$ is a filter (actually, an ultrafilter). Since $Z(f)\cap Z(h_2)=\emptyset$, we have $h_2\notin M$; thus, $h\in E_M$; since $M$ was arbitrary, $h\in E_1$. Likewise, $Z(f)\cap Z(h_1)=\emptyset$, and thus $h^{-1}=h_2/h_1\in E_1$. Hence $h$ is a unit of $E_1$, and thus $hE_1=E_1$. Therefore, $\psi$ is injective.

It follows that $\Inv(E_1)$ is isomorphic to the direct product of an infinite family of copies of $\insZ$. The latter is not free \cite[Theorem 8.2]{fuchs-abeliangroups}, and thus neither $\Inv(E)$ is free.

We note that $E$ is not strongly divisorial (see Examples 4.16-4.19 of \cite{fontana-towards}).
\end{ex}

\bibliographystyle{plain}
\bibliography{/bib/articoli,/bib/miei,/bib/libri}
\end{document}